%% file: WeeBac17.tex
\documentclass[sts]{imsart}
\input{commands.tex}
\newcommand{\RR}{\mathbb{R}}
\newcommand{\p}{\mathbb{P}}

\newtheorem{assumption}{Assumption}{\it}{\rm}
\newtheorem{applemma}{Lemma}[section]{\bf}{\it}

\newcommand*{\ep}{\varepsilon}

\DeclareMathOperator{\diam}{diam}
\DeclareMathOperator{\supp}{supp}

\begin{document}

\begin{frontmatter}
\title{Sharp asymptotic and finite-sample rates of convergence of empirical measures in Wasserstein distance}
\runtitle{Sharp rates of Wasserstein convergence}

\begin{aug}
\author{\fnms{Jonathan}~\snm{Weed}\thanksref{t1,t2}\ead[label=weed]{jweed@mit.edu}}
\and
\author{\fnms{Francis}~\snm{Bach}\thanksref{t2}\ead[label=bach]{francis.bach@inria.fr}}

\affiliation{Massachusetts Institute of Technology\\INRIA -- ENS}
\thankstext{t1}{Supported in part by the Chaire \'Economie des nouvelles donn\'ees, the data science Joint Research Initiative with the Fonds AXA pour la recherche, and the Initiative de Recherche ``Machine Learning for Large-Scale Insurance" from the Institut Louis Bachelier}
\thankstext{t2}{Supported in part by NSF Graduate Research Fellowship 1122374}

\address{{Jonathan Weed}\\
Department of Mathematics \\ 
Massachusetts Institute of Technology\\
Cambridge, MA 02139, USA\\
\printead{weed}
}

\address{{Francis Bach}\\
INRIA, D\'epartement d'informatique de l'ENS \\
CNRS \\
PSL Research University \\
75005 Paris, France\\
\printead{bach}}

\runauthor{Weed and Bach}
\end{aug}

\begin{abstract}
The Wasserstein distance between two probability measures on a metric space is a measure of closeness with applications in statistics, probability, and machine learning.
In this work, we consider the fundamental question of how quickly the empirical measure obtained from $n$ independent samples from $\mu$ approaches $\mu$ in the Wasserstein distance of any order.
We prove sharp asymptotic and finite-sample results for this rate of convergence for general measures on general compact metric spaces.
Our finite-sample results show the existence of multi-scale behavior, where measures can exhibit radically different rates of convergence as $n$ grows.
\end{abstract}
\end{frontmatter}
\section{Introduction}
The Wasserstein distance is a measure of the closeness of probability distributions on metric spaces which has proven extremely useful in data science and machine learning, particularly in the analysis of images~\cite{RubTomGui00,SolGoeCut15,SanLin11} and text~\cite{KusSunKol15,ZhaLiuLua16}.
This distance is especially useful in tasks such as classification and clustering, since it captures geometric features of the underlying data.
Moreover, unlike other measures of distance between distributions, such as the Kullback-Leibler divergence or total variation distance, the Wasserstein distance between two measures is generally finite even when neither measure is absolutely continuous with respect to the other, a situation that often arises when considering empirical distributions arising in practice.

Concretely, the Wasserstein distance measures how closely two measures can be coupled, where closeness is measured with respect to the underlying metric.
For $p \in [1, \infty)$, the Wasserstein distance of order $p$ between two distributions $\mu$ and $\nu$ on a metric space $(X, D)$ is defined as
\begin{equation*}
W_p(\mu, \nu) := \inf_{\gamma \in \cC(\mu, \nu)} \left(\int D(x, y)^p \mathrm{d}\gamma(x, y)\right)^{1/p}\,,
\end{equation*}
where the infimum is taken over all \emph{couplings} $\gamma$ of $\mu$ and $\nu$, that is, distributions on $X \times X$ whose first and second marginals agree with $\mu$ and $\nu$, respectively~\cite{Kan42}.
It can be shown that $W_p$ is a metric on the space of probability measures on~$X$~\cite[Chapter 6]{Vil09}.

In statistical contexts, direct access to a distribution of interest $\mu$ is generally not available; instead, the statistician has access to i.i.d. samples from $\mu$, or, equivalently, to an empirical distribution $\hat \mu_n$.
For $\hat \mu_n$ to serve as a reasonable proxy to $\mu$, we should insist that $\hat \mu_n$ and $\mu$ are close in the Wasserstein sense.
In the large-$n$ limit, this is indeed the case: if $X$ is compact and separable and $\mu$ is a Borel measure, then for any $p \in [1, \infty)$,
\begin{equation*}
W_p(\mu, \hat \mu_n) \to 0 \text{ $\mu$-a.s.}
\end{equation*}
This result follows from the fact that Wasserstein distances metrize weak convergence~\cite[Corollary~6.13]{Vil09} and the fact that empirical measure $\hat \mu_n$ converges weakly to $\mu$ almost surely~\cite{Var58}.

This raises the question of quantifying the rate of convergence of $\hat \mu_n$ to $\mu$ in $W_p$ distance either in expectation or with high probability.
This question is closely related to the \emph{optimal quantization} problem~\cite{GraLus07}, which asks how well a given distribution $\mu$ can be approximated by a discrete distribution with finite support, such as the empirical measure $\hat \mu_n$.
This problem has wide applications in information theory, under the name rate distortion~\cite{Sha60,cover2012elements}; machine learning~\cite{GuiRos12,NovEst14}; and numerical methods~\cite{CorPag15}.
Unfortunately, like many statistics and optimization problems involving measures on~$\RR^d$, the convergence of $\hat \mu_n$ to $\mu$ exhibits the so-called ``curse of dimensionality''~\cite{bellman1961adaptive}.
In the high-dimensional regime, the empirical distribution $\hat \mu_n$ becomes less and less representative as $d$ becomes large~\cite{friedman2001elements}, so that in the convergence of $\hat \mu_n$ to $\mu$ in Wasserstein distance is slow.

This curse of dimensionality seems unavoidable.
It was noted by Dudley~\cite{Dud68} that any measure $\mu$ that is absolutely continuous with respect to the Lebesgue measure on $\RR^d$ satisfies
\begin{equation*}
\E[W_1(\mu, \hat \mu_n)] \gtrsim n^{-1/d}\,.
\end{equation*}
This lower bound is asymptotically tight: Dudley showed that, when $d > 2$, a compactly supported measure on $\RR^d$ satisfies
\begin{equation*}
\E[W_1(\mu, \hat \mu_n)] \lesssim n^{-1/d}\,.
\end{equation*}
These results have been sharpened over the years, culminating in a tight almost sure limit theorem due to Dobri\'c and Yukich~\cite{DobYuk95}.

In short, these arguments establish that a $d$-dimensional measure yields a convergence rate in the $W_1$ distance of exactly $n^{-1/d}$.
These results are in a sense disappointing, since they show that slow convergence is a necessary price to pay for high-dimensional data.
However, they raise several questions about the behavior of the Wasserstein metric in practice:
\begin{itemize}
\item When can faster rates be achieved for measures that are not absolutely continuous with respect to the Lebesgue measure?
\item Under what conditions can sharper finite-sample (i.e., non-asymptotic) rates be obtained?
\end{itemize}

Our goal in this work is to answer the above questions in a very general sense.
We consider a bounded metric space $X$ subject to mild technical conditions and prove upper and lower bounds on the rate of convergence for $W_p(\mu, \hat \mu_n)$ for all $p \in [1, \infty)$.
Inspired by the bounds of~\cite{Dud68}, we show essentially tight asymptotic convergence rates for a large class of measures.
In particular, our upper and lower bounds improve on many existing results in the literature~\cite{Dud68,BoiLeg14,DerSchSch13,FouGui15}, either in the generality with which they are applicable or the rates which are obtained.
These results show that the rate of convergence of $W_p(\mu, \hat \mu_n)$ depends on a notion of the intrinsic dimension of the measure $\mu$, which can be significantly smaller than the dimension of the metric space on which $\mu$ is defined.

Our second goal is to obtain finite-sample results which hold outside the asymptotic regime.
A common phenomenon in practice is for a measure to exhibit different dimensional structure at different scales; this so-called multi-scale behavior arises in a range of applications~\cite{LitMagRos16,WakDonCho05,StaMurBig98}.
We show that the convergence of $\hat \mu_n$ to $\mu$ in $W_p$ for such measures can exhibit wildly different rates as $n$ increases.
In particular, they can enjoy a much faster convergence rate when $n$ is small than they do in the large-$n$ limit.
We illustrate this phenomenon via a number of examples inspired by measures that arise in practice.

In both of the above regimes, we consider exclusively the question of how the expectation $\E[W_p(\mu, \hat \mu_n)]$ behaves.
Controlling this quantity suffices to understand the behavior of $W_p(\mu, \hat \mu_n)$ because the Wasserstein distance concentrates very well around its expectation, a fact which we prove in Section~\ref{sec:concentration}.
Combining this observation with the bounds on we prove on $\E[W_p(\mu, \hat \mu_n)]$ yields sharp high-probability bounds.

We end by giving applications of our work to machine learning and statistics and sketch directions for future work.

\section{Preliminaries}\label{sec:prelim}
In this section, we present the mild assumptions on $X$ under which our results hold.
We also give background on Wasserstein distances and compare our results to prior work.

\subsection{Assumptions}
We are concerned with measures on a compact metric space~$X$.
The first assumption is entirely standard and allows us to avoid many measure-theoretic difficulties:
\begin{assumption}\label{assumption:polish}
The metric space $X$ is Polish, and all measures are Borel.
\end{assumption}

Since we limit ourselves to the compact case, $\diam(X)$ is necessarily finite, and for normalization purposes we assume the following.
\begin{assumption}\label{assumption:diameter}
$\diam(X) \leq 1$.
\end{assumption}
Assumption~\ref{assumption:diameter} can always be made to hold by a simple rescaling of the metric.

\subsection{Background on Wasserstein distances}
Above, we defined the Wasserstein $p$ distance between two distributions $\mu$ and $\nu$ on $(X, D)$ as
\begin{equation*}
W_p(\mu, \nu) := \inf_{\gamma \in \cC(\mu, \nu)} \left(\int D(x, y)^p \mathrm{d}\gamma(x, y)\right)^{1/p}\,.
\end{equation*}
A second definition, due to Monge~\cite{Mon81,San15}, reads as follows:
\begin{equation*}
W_p(\mu, \nu) := \inf_{T: \mu \circ T^{-1} = \nu} \left(\int D(x, T(x))^p \mathrm{d}\mu(x)\right)^{1/p}\,,
\end{equation*}
where the infimum is taken over all \emph{transports} $T: X \to X$ such that the pushforward measure $\mu \circ T^{-1}$ equals $\nu$.
In general, this infimum in the Monge definition is not attained.
This formulation has an easy geometric interpretation: the Wasserstein distance measures the cost of moving mass from the measure $\mu$ to the measure $\nu$ with respect to the metric of $X$.

The special case $W_1$, which is also known as the Kantorovich-Rubinstein distance~\cite{Vil09} or earth mover distance~\cite{RubTomGui00}, has a particularly simple dual representation:
\begin{equation}\label{eqn:w1_dual}
W_1(\mu, \nu) = \sup_{f \in \mathrm{Lip}(X)} \left| \int f \mathrm d \mu - \int f \mathrm d \nu\right|\,,
\end{equation}
where the supremum is taken over all $1$-Lipschitz functions on $X$~\cite{KanRub58}.
This dual representation makes $W_1$ significantly easier to bound~\cite[Remark~6.6]{Vil09}.
A more general dual formulation is also available for $W_p$ for $p \neq 1$, but it is less simple to manipulate; more details appear in Section~\ref{sec:concentration}.

\subsection{Related work}
Our work generalizes several strands of work on the convergence rates of the empirical measure in Wasserstein distances.
The first strand, inaugurated by Dudley~\cite{Dud68}, focuses on obtaining rates of convergence of $\hat \mu_n$ to $\mu$ based on the inherent dimension of the measure $\mu$.
In that paper, Dudley obtained results matching the ones we present in Section~\ref{sec:asymptotic} for the convergence of $\hat \mu_n$ to $\mu$ in $W_1$ distance, with a rate depending on the covering number of the support of~$\mu$.
Dudley's argument relied extensively on the dual characterization of $W_1$ as a supremum over Lipschitz test functions, as in~\eqref{eqn:w1_dual}.
As a result, his technique does not extend to $W_p$ for $p \neq 1$.

An extension of Dudley's techniques to other values of $p$ appears in~\cite{BoiLeg14}.
Their approach is similar to ours, but our analysis is tighter: in the language of Section~\ref{subsec:other_dimensions}, they prove an upper bound based on the quantity $d_M$ whereas we obtain an upper bound based on the smaller quantity $d_p^*$.

A second strand~\cite{FouGui15,DerSchSch13} focuses on measures on $\RR^d$ and obtains upper bounds on the rate of convergence of $\hat \mu_n$ to $\mu$ in $W_p$ for all $p \in (0, \infty)$.
The upper bounds arise from the construction of explicit couplings between $\hat \mu_n$ and $\mu$.
The construction of these couplings depends on the fact that $\mu$ is a measure on $\RR^d$ and does not extend easily to general metric spaces.
Moreover, the rates obtained, while tight for measures which are absolutely continuous with respect to the Lebesgue measure on $\RR^d$, are not tight in general, as we show below.
Nevertheless, the techniques employed in~\cite{FouGui15,DerSchSch13} are very similar to those employed in~\cite{BoiLeg14}, and we follow the same approach.

We also note several other recent works~\cite{Boi11,BolGuiVil07} which have focused on obtaining tail bounds for the quantity $W_p(\mu, \hat \mu_n)$.
The arguments of~\cite{BolGuiVil07} rely on transportation inequalities such as the celebrated Bobkov-G\"{o}tze inequality~\cite{BobGoe99}.
These arguments were simplified in~\cite{Boi11}, but, as noted in~\cite{BoiLeg14}, the analysis becomes much easier if the development of tail bounds is divided into two steps: an estimate of the expectation $\E[W_p(\mu, \hat \mu_n)]$ and a concentration bound showing how well $W_p(\mu, \hat \mu_n)$ concentrates near that expectation.
This is the approach we adopt: bounds on the expected value appear in Section~\ref{sec:asymptotic} and~\ref{sec:finite}, and concentration bounds are obtained in Section~\ref{sec:concentration}.

\subsection{Notation}
The metric on $X$ will always be denoted $D(\cdot, \cdot)$.
Given a point $x \in X$ and $r > 0$, denote by~$B(x, r)$ the open ball of radius $r$ around $x$.
The symbol $\log$ denotes the natural logarithm.
The notation $f(n) \lesssim g(n)$ indicates that there exists a constant $C$, depending on $f$ and $g$ but not $n$, such that $f(n) \leq C g(n)$ for all $n$.

\section{Dyadic transport}\label{sec:dyadic}
In order to prove upper bounds for the Wasserstein distance, we show how to construct an efficient transport between two measures based on a recursive partitioning of the underlying space.
By analogy with the dyadic intervals in $\RR$, we seek a sequence of partitions of a set such that each partition is a refinement of the last, and such that the elements of the $k$th partition have diameter of order~$\delta^k$ for some $\delta$.

We formalize these requirements in the following definition~\cite[Section~A]{Dav88}.
Denote by~$\cB(X)$ the Borel subsets of $X$.
\begin{definition}\label{def:dyadic}
A \emph{dyadic partition} of a set $S \subseteq X$ with parameter $\delta < 1$ is a sequence~$\{\cQ^k\}_{1 \leq k \leq k^*}$ with $\cQ^k \subseteq \cB(X)$ possessing the following properties:
\begin{itemize}
\item The sets in $\cQ^k$ form a partition of $S$.
\item If $Q \in \cQ^k$, then $\diam(Q) \leq \delta^k$.
\item If $Q^{k+1} \in \cQ^{k+1}$ and $Q^k \in \cQ^k$, then either $Q^{k+1} \subseteq Q^k$ or $Q^{k+1} \cap Q^k = \emptyset$. That is, the $(k+1)$th partition is a refinement of the $k$th partition.
\end{itemize}
\end{definition}

The following Proposition bounds $W_p^p(\mu, \nu)$ in terms of the mass $\mu$ and $\nu$ assign to elements of a dyadic partition.
\begin{proposition}\label{prop:wp_dyadic_bound}
Let $\mu$ and $\nu$ be two Borel probability measures on $X$, and let $S$ be a set such that $\mu(S) = \nu(S) = 1$.
If $\{\cQ^k\}_{1 \leq k \leq k^*}$ is a dyadic partition of~$S$ with parameter $\delta$, then
\begin{equation*}
W_p^p(\mu, \nu) \leq \delta^{k^* p} + \sum_{k = 1}^{k^*} \delta^{(k-1)p} \sum_{Q_i^k \in \cQ^k} |\mu(Q_i^k) - \nu(Q_i^k)|\,.
\end{equation*}
\end{proposition}
The upper bound in Proposition~\ref{prop:wp_dyadic_bound} arises from the explicit construction of a coupling between $\mu$ and $\nu$.
Proposition~\ref{prop:wp_dyadic_bound} is not new and appears to have been rediscovered many times.
In particular, it is implicit in the proof of~\cite[Proposition~1.1]{BoiLeg14}, and similar results have appeared before in papers bounding the convergence of $W_p^p$ when $X = \RR^d$~\cite{DerSchSch13,FouGui15}.
An analogous bound has also been used in the computer science community~\cite{Ind03,BaNguNgu11}.
The idea of bounding the quantity $W_p^p(\mu, \nu)$ by considering the mass each measure assigns to elements of a sequence of partitions is present also in~\cite{AjtKomTus84}, where it is used to obtain sharp results for the case $X = [0, 1]^2$.
We include a proof in Appendix~\ref{appendix:omitted} for clarity and because we could not find a suitably general version explicitly stated in the literature.

Proposition~\ref{prop:wp_dyadic_bound} is stated for $W_p^p$, but can easily adapted to optimal transport with a general cost $c(\cdot, \cdot)$ by replacing the requirement that $\diam(Q) \leq \delta^k$ in Definition~\ref{def:dyadic} by the requirement that $\sup_{x, y \in Q} c(x, y) \leq \delta^k$.

Boissard and Le Gouic~\cite{BoiLeg14} used a version of Proposition~\ref{prop:wp_dyadic_bound} to prove a bound on $W_p^p(\mu, \hat \mu_n)$ based on the \emph{covering number} of the set $S$, a definition of which appears in Section~\ref{sec:asymptotic}, below.
However, their results are not sharp, and they do not recover the rates obtained in~\cite{Dud68} for the case $p = 1$.
In Section~\ref{sec:asymptotic}, we show how to improve their argument to obtain sharper results, which extend the rates from~\cite{Dud68} to all $p \in [1, \infty)$.

\section{Asymptotic upper and lower bounds}\label{sec:asymptotic}
In this section, we show asymptotic upper and lower bounds for $W_p$ that hold for all $p \in [1, \infty)$.
These bounds extend results of~\cite{Dud68} to the case $p \neq 1$ and improve the bounds of~\cite{BoiLeg14} by focusing on a set $S$ to which $\mu$ assigns mass of \emph{almost} $1$ rather than on the larger set $\supp(\mu)$.
We will also show a broad class of measures for which our bounds are asymptotically tight.

\subsection{Definitions}To state our bounds we will define several notions of dimension of a measure.
\begin{definition}
Given a set $S \subseteq X$, the \emph{$\ep$-covering number of $S$}, denoted~$\cN_\ep(S)$, is the minimum $m$ such that there exists $m$ closed balls $B_1, \dots, B_m$ of diameter~$\ep$ such that $S \subseteq \bigcup_{1 \leq i \leq m} B_i$.
The \emph{$\ep$-dimension of $S$} is the quantity
\begin{equation*}
d_\ep(S) := \frac{\log \cN_\ep(S)}{- \log \ep}\,.
\end{equation*}
\end{definition}

When working with measures instead of sets, it is convenient to be able to ignore a small fraction of the mass.
The following definition appears in~\cite{Dud68}, which notes a connection to the $\ep;\delta$ entropy introduced by~\cite{PosRodRum67}.
\begin{definition}
Given a measure $\mu$ on $X$, the \emph{($\ep$, $\tau$)-covering number} is
\begin{equation*}
\cN_\ep(\mu, \tau) := \inf \{\cN_\ep(S): \mu(S) \geq 1- \tau\}
\end{equation*}
and the \emph{($\ep$, $\tau$)-dimension} is
\begin{equation*}
d_\ep(\mu, \tau) := \frac{\log \cN_\ep(\mu, \tau)}{-\log \ep}\,.
\end{equation*}
\end{definition}

For convenience, let
\begin{align*}
\cN_\ep(\mu) & := \cN_\ep(\mu, 0)\,, \\
d_\ep(\mu) & := d_\ep(\mu,0)\,.
\end{align*}

Note that $\cN_\ep(\mu) = \cN_\ep(\supp(\mu))$, and that $\cN_\ep(\mu, \tau)$ and $d_\ep(\mu, \tau)$ increase as $\tau$ decreases.

We now define our main notions of dimension of a measure.
\begin{definition}
The \emph{upper and lower Wasserstein dimensions} are respectively
\begin{align*}
d_p^*(\mu) & = \inf \{s \in (2p, \infty): \limsup_{\ep \to 0} d_\ep(\mu, \ep^{\frac{sp}{s-2p}}) \leq s\}\,, \\
d_*(\mu) & = \lim_{\tau \to 0} \liminf_{\ep \to 0} d_\ep(\mu, \tau)\,.
\end{align*}
\end{definition}
Note that the monotonicity of $d_\ep(\mu, \tau)$ in $\tau$ implies that the limit in the definition of $d_*(\mu)$ exists.
The definition of $d_p^*$ is complicated by the fact that the behavior of the Wasserstein distance is very different when the dimension is small.
For convenience we only treat the case where the dimension is larger than $2p$.
We note that the monotonicity of $d_\ep(\mu, \tau)$ in $\tau$ also implies that $d_* \leq d_p^*$ for all $p$.

Our definition of the upper Wasserstein dimension is new.
Dudley~\cite{Dud68} considered measures satisfying a bound of the form $\cN_\ep(\mu, \ep^{\frac{s}{s-2}}) \leq C \ep^{-s}$ for all sufficiently small $\ep$; the definition of $d_p^*(\mu)$ is the correct generalization to the $p \neq 1$ case.
The lower Wasserstein dimension was introduced by Young~\cite{You82}, who credits the idea to Ledrappier~\cite{Led81}, in the context of dynamical systems.
The term Wasserstein dimension is ours, and is justified by Theorem~\ref{thm:main_asymptotic} below.

\subsection{Comparison with other notions of dimension}\label{subsec:other_dimensions}
To make it easier to interpret the quantities $d_p^*(\mu)$ and $d_*(\mu)$, we sketch here their relationship with two other well known notions of dimensions for the measure $\mu$, the Minkowski dimension (also known as the Minkowski-Bouligand or box-counting dimension) and the Hausdorff dimension.
Both quantities have long been studied in fractal and metric geometry~\cite{falconer2004fractal}.
\begin{definition}
The \emph{Minkowski dimension} of a set $S$ is the quantity
\begin{equation*}
\dim_M(S) := \limsup_{\ep \to 0} d_\ep(S)\,.
\end{equation*}
The \emph{$d$-Hausdorff measure} of $S$ is
\begin{equation*}
\cH^d(S) = \lim_{\ep \to 0} \inf\left\{\sum_{k=1}^\infty r_k^d: S \subseteq \bigcup_{k=1}^\infty B(x_i, r_i); \,\,r_k \leq \ep \, \forall k\right\}\,,
\end{equation*}
and its \emph{Hausdorff dimension} is
\begin{equation*}
\dim_H(S) := \inf \{d: \cH^d(S) = 0\}\,.
\end{equation*}
Given a measure $\mu$, the \emph{Minkowski and Hausdorff dimensions} of $\mu$ are respectively
\begin{align*}
d_M(\mu) &:= \inf \{\dim_M(S): \mu(S) = 1\}\,, \\
d_H(\mu) &:= \inf \{\dim_H(S): \mu(S) = 1\}\,.
\end{align*}
\end{definition}

We note that the quantities $d_M(\mu)$ and $d_H(\mu)$ are upper and lower bounds on the Wasserstein dimensions.
\begin{proposition}\label{prop:ordering}
\begin{equation*}
d_H(\mu) \leq d_*(\mu) \leq d_p^*(\mu)\,.
\end{equation*}
If $d_M(\mu) \geq 2p$, then
\begin{equation*}
d^*_p(\mu) \leq d_M(\mu)\,.
\end{equation*}
\end{proposition}
A proof appears in Appendix~\ref{appendix:omitted}.

None of the inequalities in Proposition~\ref{prop:ordering} can be replaced by equalities.
Examples of measures $\mu$ for which $d_H(\mu) < d_*(\mu)$ are complicated; one appears in~\cite[Remark~7.8]{KaiPri11}.
It is much easier to find examples in which $d_*(\mu)$, $d^*_p(\mu)$, and $d_M(\mu)$ do not agree.
For instance, it is easy to see that $d_*(\mu) = 0$ for any discrete measure, but the countable set $S := \left(\{k^{-1}\}_{k=1}^\infty\right)^d \subset [0, 1]^d$ has Minkowski dimension $d/2$.
By choosing $d > 4p$ and choosing a measure $\mu$ supported on $S$ with appropriately slow decay, one can ensure that $d^*_p(\mu)$ is strictly less than $d/2$, and hence strictly between $d_*(\mu)$ and $d_M(\mu)$.

\subsection{Main result}
With these definitions in place, we can state our main asymptotic bound.
\begin{theorem}\label{thm:main_asymptotic}
Let $p \in [1, \infty)$.
If $s > d_p^*(\mu)$, then
\begin{equation*}
\E[W_p(\mu, \hat \mu_n)] \lesssim n^{-1/s}\,.
\end{equation*}
If $t < d_*(\mu)$, then
\begin{equation*}
W_p(\mu, \hat \mu_n) \gtrsim n^{-1/t}\,.
\end{equation*}
\end{theorem}
The upper and lower bounds are proved below and are corollaries of more precise results with explicit constants (Propositions~\ref{prop:relaxed_upper_bound} and~\ref{prop:lossy_lower}).
Note that the lower bound does not merely hold in expectation.
Indeed, such a lower bound holds for \emph{any} discrete measure supported on at most $n$ points.

Theorem~\ref{thm:main_asymptotic} improves on several existing results.
For the upper bound, Dudley~\cite{Dud68} showed that, if $s > d_1^*(\mu)$, then
\begin{equation*}
\E[W_1(\mu, \hat \mu_n)] \lesssim n^{-1/s}\,,
\end{equation*}
but his proof technique applied only to $p = 1$.
Boissard and Le Gouic~\cite{BoiLeg14} extended this bound to all $p$, but only if $s > d_M(\mu) \geq 2p$.
Since $d^*_p(\mu) \leq d_M(\mu)$ with some measures exhibiting strict inequality, our result is sharper.

Dudley~\cite{Dud68} proved a lower bound for $W_1$---and hence, by monotonicity of~$W_p$ in $p$, for $W_p$ for all $p \in [1, \infty)$---based on the quantity
\begin{equation*}
d_{1/2}(\mu) = \liminf_{\ep \to 0} \frac{\log \cN_\ep(\mu, 1/2)}{- \log \ep}\,,
\end{equation*}
which is easily seen to be smaller than $d_*(\mu)$, with strict inequality possible.
Our argument is a simple extension of his.

\subsection{Proof of upper bound}
The upper bound of Theorem~\ref{thm:main_asymptotic} follows from Proposition~\ref{prop:relaxed_upper_bound}, below.

To apply the bound of Proposition~\ref{prop:wp_dyadic_bound}, we need to show the existence of a suitable dyadic partition.
The following Proposition is an extension of~\cite[Lemma~2.1]{BoiLeg14} and shows that we can choose a dyadic partition which provides an almost optimal covering of subsets of $S$.

\begin{proposition}\label{prop:covering_to_dyadic}
Fix $S \in \cB(X)$.
Let $k^*$ be any positive integer for which the covering number $\cN_{3^{-(k^*+1)}}(S)$ is finite, and let $\{S_k\}_{1 \leq k \leq k^*}$ be a sequence of Borel subsets $S$.
There exists a dyadic partition of $S$ with parameter $\delta = 1/3$ such that for $1 \leq k \leq k^*$, the number of sets in $\cQ^k$ intersecting $S_k$ is at most $\cN_{3^{-(k+1)}}(S_k)$.
\end{proposition}
A proof appears in Appendix~\ref{appendix:omitted}.

All the upper bounds we prove rely on the following fundamental estimate, which was used in~\cite{FouGui15} to provide bounds in the case where $X = \RR^d$.

\begin{proposition}\label{prop:expectation_estimate}
If $S$ is any Borel set, then
\begin{equation*}
\E\left[\sum_{Q_i^k \in \cQ^k} |\mu(Q_i^k) - \hat \mu_n(Q_i^k)|\right] \leq 2 (1 - \mu(S)) + \sqrt{|\{i: Q_i^k \cap S \neq \emptyset\}|/n}\,.
\end{equation*}
\end{proposition}
\begin{proof}
Let $Q(S) = \{i: Q_i^k \cap S \neq \emptyset\}$, and write
\begin{equation*}
S' = \bigcup_{i \in Q(S)} Q_i^k\,.
\end{equation*}

Since $n \hat \mu_n(Q_i^k)$ is a Binomial random variable with parameters $(n, \mu(Q_i^k))$, we have the bound~\cite{BerKon13}
\begin{equation*}
\E|\mu(Q_i^k) - \hat \mu_n(Q_i^k)| \leq \sqrt{\mu(Q_i^k)/n} \wedge 2 \mu(Q_i^k)\,.
\end{equation*}

Applying the first bound on $S'$ and the second bound on $X \setminus S'$ yields
\begin{equation*}
\E\left[\sum_{Q_i^k \in \cQ^k} |\mu(Q_i^k) - \hat \mu_n(Q_i^k)|\right] \leq 2 \mu(X \setminus S') + \sum_{i \in Q(S)} \sqrt{\mu(Q_i^k)/n}\,.
\end{equation*}
Since the second sum contains $|Q(S)|$ terms and $\sum_{i \in Q(S)} \mu(Q_i^k) = \mu(S') \leq 1$, the final bound follows from Cauchy-Schwarz.
\end{proof}

The key step in proving the upper bound of Theorem~\ref{thm:main_asymptotic} is giving a bound for $\E[W_p^p(\mu, \hat \mu_n)]$ in terms of the quantity $d_\ep(\mu, \ep^{\frac{sp}{s - 2p}})$, which appears in the definition of $d_p^*$.
\begin{proposition}\label{prop:relaxed_upper_bound}
Let $p \in [1, \infty)$.
Suppose there exists an $\ep' \leq 1$ and $s > 2p$ such that
\begin{equation*}
d_{\ep}(\mu, \ep^{\frac{sp}{s - 2p}}) \leq s
\end{equation*}
for all $\ep \leq \ep'$. Then
\begin{equation*}
\E[W_p^p(\mu, \hat \mu_n)] \leq C_1 n^{-p/s} + C_2 n^{-1/2}\,,
\end{equation*}
where
\begin{equation*}
C_1 = 3^{\frac{3sp}{s - 2p}+1}\left(\frac{1}{3^{\frac s 2 - p} - 1} + 3\right)\,, \text{ and } C_2  = (27/\ep')^{\frac s 2}\,.
\end{equation*}
In particular,
\begin{equation*}
\E[W_p(\mu, \hat \mu_n)]  \lesssim n^{-1/s}\,.
\end{equation*}
\end{proposition}
The assumption that $s > 2p$ implies that the first term in the above bound is asymptotically larger than the second term.
Note also that $C_1$ decreases as $s$ increases, so that as long as $s$ is bounded away from $2p$, the constant $C_1$ has no dependence on the dimension $s$.
On the other hand, $C_2$ does depend exponentially on $s$, even though the term $C_2 n^{-1/2}$ is asymptotically negligible.

The presence of two terms in the upper bound of Proposition~\ref{prop:relaxed_upper_bound} is a consequence of the weakness of the assumption that the bound on $d_\ep$ holds only for~$\ep$ sufficiently small rather than for all $\ep$.
In Proposition~\ref{prop:finite_sample}, below, we remove the $n^{-1/2}$ term by adopting a stronger assumption on $d_\ep$.

\begin{proof}
If $n < (27/\ep')^s$, then the second term is larger than $1$, so the bound holds from the trivial fact that $W^p_p(\mu, \nu) \leq \diam(X) \leq 1$ for any measures $\mu, \nu$ supported on~$X$.
We therefore assume that $n \geq (27/\ep')^s$.

For convenience, write $\alpha = sp/(s-2p)$ and $\ell = \lceil \frac{- \log \ep'}{\log 3} \rceil$.
Let $k^* = \left \lfloor \frac{\log n}{s \log 3}\right \rfloor - 2$.
Let $k'$ be the largest integer in the range $[\ell, k^*]$ satisfying $k' \leq \frac{p}{\alpha} \cdot \frac{\log n}{s \log 3}$, or $\ell$ if no such integer exists.

Our assumptions imply that for all $k \geq \ell$,
\begin{equation*}
\cN_{3^{-k}}(\mu, 3^{-\alpha k}) \leq 3^{k s}\,.
\end{equation*}
Hence for $k \geq k'$, there exists a set $T_k$ of mass at least $1 - 3^{-\alpha k'}$ such that
\begin{equation*}
\cN_{3^{-k}}(T_k) \leq 3^{k s}\,.
\end{equation*}
Applying Proposition~\ref{prop:covering_to_dyadic} with $S_k = T_{k'}$ for $k < k'$ and $S_k = T_{k+1}$ for $k \geq k'$ implies the existence of a dyadic partition~$\{\cQ^k\}_{1 \leq k \leq k^*}$ of $X$ such that the number of sets of $\cQ^k$ intersecting $S_k$ is at most $\cN_{3^{-(k+1)}}(S_k)$.

Using this dyadic partition in Proposition~\ref{prop:wp_dyadic_bound} and applying Proposition~\ref{prop:expectation_estimate} yields
\begin{align*}
\E[W^p_p(\mu, \hat \mu_n)] \leq 3^{-k^* p} & + \sum_{k=1}^{k'-1} 3^{-(k-1)p} \sqrt{\frac{\cN_{3^{-(k+1)}}(T_{k'})}{n}} \\
 & + \sum_{k=k'}^{k^*} 3^{-(k-1)p} \sqrt{\frac{\cN_{3^{-(k+1)}}(T_{k+1})}{n}}\\
 &+ 2 \cdot 3^{-\alpha k'} \sum_{k=1}^{k^*}3^{-(k-1)p}\,.
\end{align*}

Since $\cN_\ep(T)$ increases as $\ep$ decreases, for $k \leq k' -1$ we have the bound
\begin{equation*}
\cN_{3^{-(k+1)}}(T_{k'}) \leq \cN_{3^{-k'}}(T_{k'}) \leq 3^{k' s}\,.
\end{equation*}
By construction, the sets $T_k$ also satisfy for $k \geq k'$
\begin{equation*}
\cN_{3^{-(k+1)}}(T_{k+1}) \leq 3^{(k+1)s}\,.
\end{equation*}

Combining these bounds with the bound $\sum_{k=1}^{k^*} 3^{-(k-1)p} \leq 3/2$ for $p \geq 1$ yields
\begin{align*}
\E[W^p_p(\mu, \hat \mu_n)] & \leq 3^{-k^* p} + \frac 3 2\left(\frac{3^{k' s/2}}{\sqrt n} + 2 \cdot 3^{-\alpha k'}\right) + 3^{2p}\sum_{k = k'}^{k^*} \frac{3^{(k +1)(\frac s 2  -p)} }{\sqrt n} \\
& \leq 3^{-k^* p} + \frac 3 2 \left(\frac{3^{k' s/2}}{\sqrt n} + 2 \cdot 3^{-\alpha k'}\right) + \frac{3^{-k^* p}}{3^{\frac s 2 - p} - 1}\sqrt{\frac{3^{(k^*+2)s}}{n}}\,.
\end{align*}

The choice of $k^*$ implies that $3^{(k^* + 2)s} \leq n$ and that $3^{-k^* p} \leq 3^{3p} n^{-p/s}$, and the choice of $k'$ implies that $\alpha k' > p \frac{\log n}{s \log 3} - 3\alpha$, so that $3^{- \alpha k'} < 3^{3\alpha} n^{-p/s}$.
Combining these estimates yields
\begin{equation*}
\E[W^p_p(\mu, \hat \mu_n)] \leq \left(3^{3p}+\frac{3^{3p}}{3^{\frac s 2 - p} -1} + 3^{3\alpha + 1}\right) n^{-p/s} + \frac{3 \cdot 3^{k' s/2}}{2\sqrt n}\,.
\end{equation*}

The definition of $k'$ implies that $sk' \leq \max\{s\ell, (\frac{p}{\alpha} \cdot \frac{\log n}{\log 3})\}$, so
\begin{equation*}
3^{k' s/2} \leq 3^{\ell s/2} + n^{p/2\alpha} = 3^{\ell s/2} + n^{1/2}n^{-p/s}\,.
\end{equation*}

Plugging in the definitions of $C_1$ and $C_2$ then yields the claim.
\end{proof}

\begin{corollary}\label{cor:minkowski_upper_bound}
If $s > d_p^*(\mu)$, then
\begin{equation*}
\E[W_p(\mu, \hat \mu_n)] \lesssim n^{-1/s}\,.
\end{equation*}
\end{corollary}
\begin{proof}
If $s > d_p^*(\mu)$, then there exists an $\ep'$ such that $d_{\ep}(\mu, \ep^{\frac{sp}{s - 2p}k}) \leq s$ for all $\ep \leq \ep'$.
Apply Proposition~\ref{prop:relaxed_upper_bound}.
\end{proof}

\subsection{Proof of lower bound}
Our asymptotic lower bounds involving $d_*(\mu)$ follow from a much simpler argument.
One striking feature of this lower bound is that it actually holds not merely for the empirical measure $\hat \mu_n$ but indeed for \emph{any} measure $\nu$ supported on at most $n$ atoms.
That such lower bounds are often tight for empirical measures is a rather surprising fact, which has been noted several times, including in Dudley's original paper~\cite{GuiRos12,Klo12,DerSchSch13,Dud68,BoiLeg14}.

The following Proposition is adapted from~\cite{Dud68} and forms the core of the lower bound.
\begin{proposition}\label{prop:lossy_lower}
Suppose that there exist positive constants $\ep'$, $\tau$, and $t$ such that
\begin{equation*}
\cN_\ep'(\mu, \tau) \geq \ep^{-t}
\end{equation*}
for all $\ep \leq \ep'$.
If $n > {\ep'}^{-t}$ and $\nu$ is any measure supported on at most $n$ points, then
\begin{equation*}
W^p_p(\mu, \nu) \geq \tau 4^{-p} n^{-p/t}\,.
\end{equation*}
\end{proposition}
\begin{proof}
Choose $\ep = n^{-1/t}/2$, and let $S = \bigcup_{x \in \supp(\nu)} B(x, \ep/2)$.
Since $\cN'_\ep(\mu, \tau) \geq \ep^{-t} > n$, we must have $\mu(S) < 1 - \tau$.
Therefore, if $X \sim \mu$, then $D(X, \supp(\nu)) \geq \ep/2$ with probability at least $\tau$.
Hence if $(X, Y)$ is any coupling of $\mu$ and $\nu$,
\begin{equation*}
\E[D(X, Y)^p] \geq \E[D(X, \supp(\nu))^p] \geq \tau (\ep/2)^p = \tau 4^{-p} n^{-p/t}\,.
\end{equation*}
\end{proof}

Proposition~\ref{prop:lossy_lower} immediately implies the desired asymptotic lower bound.
\begin{corollary}\label{cor:lower_bound}
If $t < d_*(\mu)$ and $\nu$ is any measure supported on at most $n$ points, then
\begin{equation*}
W_p(\mu, \nu) \gtrsim n^{-1/t}\,.
\end{equation*}
\end{corollary}
\begin{proof}
By the definition of $d_*(\mu)$, for any $t < d_*(\mu)$, there exist constants $\ep'$ and $\tau$ as in the statement of Proposition~\ref{prop:lossy_lower}.
The claim follows.
\end{proof}
\subsection{Regular spaces}
The remark after Proposition~\ref{prop:ordering} establishes that $d_*(\mu)$ and $d^*_p(\mu)$ do not agree in general.
However, these dimensions do agree whenever the measure is sufficiently well behaved.
In this section, we give several broad classes of examples for which they do match, and for which our bounds are therefore sharp.

The following Proposition gives a simple condition under which this agreement occurs.
\begin{proposition}\label{prop:abs_cont}
Let $\cH^d$ be the $d$-dimensional Hausdorff measure on a closed set $S$.
If $\mu \ll \cH^d$ and $\supp(\mu) \subseteq S$, then for any $p \in [1, d/2]$,
\begin{equation*}
d \leq d_*(\mu) \leq d^*_p(\mu) \leq d_M(S)\,.
\end{equation*}
In particular, if $d = d_M(S)$, then $d_*(\mu) = d^*_p(\mu) = d$.
\end{proposition}
A proof appears in Appendix~\ref{appendix:omitted}.

Proposition~\ref{prop:abs_cont} immediately implies the result quoted in the Introduction (up to subpolynomial factors): since the set $[0,1]^d$ satisfies $d = d_M([0, 1]^d)$, Theorem~\ref{thm:main_asymptotic} implies that any measure $\mu$ absolutely continuous with respect to the Lebesgue measure on $\RR^d$ (or, equivalently, to $\cH^d$) must satisfy
\begin{equation*}
n^{-1/t} \lesssim \E[W_p(\mu, \hat \mu_n)] \lesssim n^{-1/s}
\end{equation*}
for any $t < d < s$ and $p \in [1, d/2]$.

Limiting our attention to sets for which the Hausdorff measure is well behaved motivates the following definition, which appears in~\cite{GraLus07}.
\begin{definition}
A set $S$ is \emph{regular of dimension $d$} if it is compact and there exists constants $c$ and $r_0$ such that the $d$-dimensional Hausdorff measure $\cH^d$ on $S$ satisfies
\begin{equation*}
\frac 1 c r^d \leq \cH^d(B(x, r)) \leq c r^d\,,
\end{equation*}
for all $x \in S$.
\end{definition}

It is well known (see, e.g.,~\cite[Theorem~5.7]{Mat99}) that $d_M(S) = d$ if $S$ is regular of dimension $d$.
We therefore obtain the following simple characterization.
\begin{proposition}\label{prop:regular}
If the support of $\mu$ is a regular set of dimension $d$ and $\mu \ll \cH^d$, then for any $p \in [1, d/2]$,
\begin{equation*}
d_*(\mu) = d_p^*(\mu) = d\,.
\end{equation*}
\end{proposition}

The following Proposition, which appears in~\cite{GraLus07}, shows that many well behaved sets are regular, and so implies the existence of many examples for which our results are tight.
\begin{proposition}[{\cite{GraLus07}}]
The following sets are regular of dimension $d$:
\begin{itemize}
\item Nonempty, compact convex sets spanned by an affine space of dimension~$d$,
\item Relative boundaries of nonempty, compact convex sets of dimension~$d+1$,
\item Compact $d$-dimensional differentiable manifolds,
\item Self-similar sets with similarity dimension $d$.
\end{itemize}
Moreover, regularity is preserved under finite unions and bi-Lipschitz maps.
\end{proposition}

\section{Finite-sample bounds and multiscale behavior}\label{sec:finite}
The results of Section~\ref{sec:asymptotic} imply that for any sufficiently regular $d$-dimensional measure~$\mu$, the empirical measure $\hat \mu_n$ approaches $\mu$ in $W_p$ at a rate of approximately~$n^{-1/d}$.
For example, if $\mu$ is absolutely continuous with respect to the Lebesgue measure on $[0, 1]^d$, Dudley showed that the slow $n^{-1/d}$ rate is unavoidable~\cite{Dud68}.
Faster rates can be obtained if $\mu$ is singular: for instance, if $\mu$ is a sum of a finite number of Dirac masses, then Proposition~\ref{prop:wp_dyadic_bound} can be used to show that $\hat \mu_n$ approaches $\mu$ at a much faster $n^{-1/2p}$ rate, independent of the ambient dimension.

However, what should one expect if $\mu$ is \emph{approximately} a sum of Dirac masses (or, in general, approximately low dimensional)?
Suppose for instance that $\mu$ is the convolution of a sum of Dirac masses with an isotropic Gaussian of small variance.
Since $\mu$ has a density, $W_p(\mu, \hat \mu_n)$ must scale like $n^{-1/d}$ eventually, but it is possible that the convergence of $\hat \mu_n$ to $\mu$ should improve due to the fact that $\mu$ is almost singular.

It turns out that this is indeed the case, as we show in this Section.
We begin by proving a sharper version of Proposition~\ref{prop:relaxed_upper_bound} better suited to non-asymptotic results.
In the second half of this Section, we show how this non-asymptotic bound can be used to prove faster convergence rates in the finite-sample regime for situations like the one described above.

\subsection{Finite-sample behavior}
The statement of Proposition~\ref{prop:relaxed_upper_bound} only assumes a bound on the quantity $d_\ep(\mu, \tau)$ for sufficiently small $\ep$.
It is therefore well suited to establishing results of an asymptotic nature.
On the other hand, the resulting bound did not give any indication of the behavior in the small-$n$ regime, since the bound was vacuous for $n \lesssim (\ep')^{-2}$.

If we have stronger control over $d_\ep(\mu, \tau)$, then the proof of Proposition~\ref{prop:relaxed_upper_bound} can be modified to yield a finite-sample result.
In particular, if we can control $d_{\ep}(\mu, \tau)$ for all $\ep$ larger than a certain threshold, we can prove an upper bound without the $n^{-1/2}$ term present in Proposition~\ref{prop:relaxed_upper_bound}.

\begin{proposition}\label{prop:finite_sample}
Fix $p \in [1, \infty)$.
Write $d_{\geq \ep}(\mu, \tau) = \sup_{\ep' \in [\ep, 1/9]}d_{\ep}(\mu, \tau)$, and let $d_n = \inf_{\ep > 0} \max\{d_{\geq \ep}(\mu, \ep^p), \frac{\log n}{- \log \ep}\}$.
If $d_n > 2p$, then
\begin{equation*}
\E[W^p_p(\mu, \hat \mu_n)] \leq C_1 n^{-p/d_n}\,,
\end{equation*}
where
\begin{equation*}
C_1 = 27^p\left(2+\frac{1}{3^{\frac{d_n}{2} - p} -1}\right)\,.
\end{equation*}
\end{proposition}
As in Proposition~\ref{prop:relaxed_upper_bound}, the constant $C_1$ is independent of the dimension as long as $d_n$ is bounded away from $2 p$.
\begin{proof}
Fix an arbitrary $\ep$, and let $d = \max\{d_{\geq \ep}(\mu, \ep^p), \frac{\log n}{- \log \ep}\}$, where $d > 2p$.
If $n^{-1/d} \geq 1/27$, then the bound $W^p_p(\mu, \hat \mu_n) \leq C_1 n^{-p/d}$ is trivial, so assume that $n > 3^{3d}$.

Let $k^* = \left \lfloor \frac{\log n}{d \log 3} \right \rfloor - 2$.
As in the proof of Proposition~\ref{prop:relaxed_upper_bound}, we can choose sets $S_1, \dots, S_{k^*}$ such that $\mu(S_k) \geq 1 - \ep^p$ and $\cN_{3^{-(k+1)}}(S_k) = \cN_{3^{-(k+1)}}(\mu, \ep^p)$ for $1 \leq k \leq k^*$.
Applying Proposition~\ref{prop:covering_to_dyadic} to construct an appropriate dyadic partition and using Propositions~\ref{prop:wp_dyadic_bound} and~\ref{prop:expectation_estimate} yields
\begin{equation*}
\E[W_p^p(\mu, \hat \mu_n)] \leq 3^{- k^* p} + \sum_{k=1}^{k^*}3^{-(k-1)p}\sqrt{\frac{\cN_{3^{-(k+1)}}(\mu, \ep^p)}{n}} + 2 \ep^p \sum_{k=1}^{k^*} 3^{-(k-1)p}\,.
\end{equation*}
By the definition of $k^*$, for $1 \leq k \leq k^*$,
\begin{equation*}
3^{(k+1)d} \leq n\,,
\end{equation*}
so $3^{-(k + 1)} \geq \ep$.
Hence $3 \ep^p \leq 3^{-k^* p}$ and $\cN_{3^{-(k+1)}}(\mu, \ep^p) \leq 3^{(k+1)d}$ for $1 \leq k \leq k^*$, and applying the bound $\sum_{k=1}^{k^*} 3^{-(k-1)p} \leq 3/2$ for $p \geq 1$ yields
\begin{align*}
\E[W_p^p(\mu, \hat \mu_n)] & \leq 3^{-k^* p} + \frac{3^{-k^*p}}{3^{\frac{d}{2} - p} -1} \sqrt{\frac{3^{(k^* +2)d}}{n}} + 3\ep^p \\
& \leq \left(2+\frac{1}{3^{\frac{d_n}{2} - p} -1}\right) 3^{-k^* p} \\
& \leq C_1 n^{-p/d}\,,
\end{align*}
where in the last step we have used the fact that $d \geq d_n$ and $\frac{1}{3^{\frac{d}{2} - p} -1}$ is decreasing in $d$.

Taking the infimum over all possible choices of $\ep$ yields the bound.
\end{proof}

Note in the proof of Proposition~\ref{prop:finite_sample} that we in fact only needed control over $d_{\ep'}(\mu, \tau)$ for $\ep'$ of the form $3^{-k}$ for $k$ a positive integer, though for simplicity we have assumed that we can bound $d_{\ep'}(\mu, \tau)$ for all $\ep' \in [\ep, 1/9]$.

The upper bound of Proposition~\ref{prop:finite_sample} suggests that measures can have truly different rates of convergence at different scales.
The following Proposition shows that Proposition~\ref{prop:finite_sample} is essentially tight and that this multiscale behavior indeed can occur, in the sense that for \emph{any} decreasing sequence $\delta_n$ satisfying mild conditions, there exists a measure $\mu$ such that $n^{-1/d_n} \approx \delta_n$ and $\E[W_p(\mu, \hat \mu_n)] \geq C n^{-1/d_n}$ for all $n$.
In other words, for any desired rate of decrease, there exists a measure such that $\E[W_p(\mu, \hat \mu_n)]$ converges to $0$ at precisely that rate.
Such measures can even be found when the underlying metric metric is induced by the $\ell_\infty$ norm on real space.
As with the lower bound proved in Proposition~\ref{prop:lossy_lower}, above, the following bound in fact holds for \emph{any} measure $\nu$ supported on at most $n$ points.

A proof appears in Appendix~\ref{appendix:omitted}.
\begin{proposition}\label{prop:finite_sample_converse}
Let $\delta_n$ be a nonincreasing sequence in $(0, 1)$ with the following properties:
\begin{itemize}
\item the bound $\delta_n \geq n^{-1}$ holds for all $n \geq 2$ (i.e., $\delta_n$ does not decrease too quickly)
\item the sequence $\frac{\log n}{- \log \delta_n}$ is nondecreasing (i.e., the rate of decrease of $\delta_n$ slows), and
\item there exist constants $c > 1$ and $\alpha \in [-1, 0)$ such that $\frac 1 c n^\alpha \leq \delta_n \leq c n^\alpha$ for all $n$ sufficiently large (i.e., $\delta_n$ eventually decreases polynomially in $n$).
\end{itemize}

There exists a measure~$\mu$ on $X = ([0, 1]^m, \ell_\infty)$ for some $m$ such that, if $d_n$ is defined as in Proposition~\ref{prop:finite_sample}, then $\frac 1 4 \delta_n \leq n^{-1/d_n} \leq 2 \delta_n$ and
\begin{equation*}
\E[W_p(\mu, \nu)] \geq 2^{-6} n^{-1/d_n}
\end{equation*}
for all $p \in [1, \infty)$, all $n \geq 1$, and any measure $\nu$ supported on at most $n$ points.
\end{proposition}

Proposition~\ref{prop:finite_sample} only holds when $d_n > 2p$, so for completeness we conclude this section by providing a second bound that can be used when Prposition~\ref{prop:finite_sample} does not apply.
The following bound is always valid and is sharper when the asymptotic dimension of $\mu$ is small.

\begin{proposition}\label{prop:sqrt_n_bound}
Let $m_n = \inf_{\ep > 0}\max\{\cN_{\ep}(\mu, \ep^p), n \ep^{2p}\}$.
Then
\begin{equation*}
\E[W_p^p(\mu, \hat \mu_n)] \leq C_1 \sqrt{\frac{m_n}{n}}\,,
\end{equation*}
where $C_1 = 9^p+3$.
\end{proposition}
\begin{proof}
Fix an arbitrary $\ep$, and let $m = \max\{\cN_{\ep}(\mu, \ep^p), n \ep^{2p}\}$.
If $\ep \geq 1/9$, then the bound $W_p^p(\mu, \hat \mu_n) \leq C_1 \sqrt{\frac m n}$ is trivial, so assume $\ep < 1/9$.

Let $k^* = \Big\lfloor \frac{- \log \ep}{\log 3} \Big\rfloor - 1$.
Following the proof of Proposition~\ref{prop:finite_sample}, we have
\begin{equation*}
\E[W_p^p(\mu, \hat \mu_n)] \leq 3^{- k^* p} + \sum_{k=1}^{k^*}3^{-(k-1)p}\sqrt{\frac{\cN_{3^{-(k+1)}}(\mu, \ep^p)}{n}} + 2 \ep^p \sum_{k=1}^{k^*} 3^{-(k-1)p}\,.
\end{equation*}
The monotonicity of $\cN_\ep(\mu, \tau)$ implies that $\cN_{3^{-(k+1)}}(\mu, \ep^p) \leq \cN_{\ep}(\mu, \ep^p) \leq m$ for all $k \leq k^*$.
Plugging in this estimate and applying the bound $\sum_{k=1}^{k^*} 3^{-(k-1)p} \leq 3/2$ for $p \geq 1$ yields
\begin{equation*}
\E[W_p^p(\mu, \hat \mu_n)] \leq 3^{-k^* p} + \frac 3 2 \sqrt{\frac{m}{n}} + \frac 3 2 \ep^p \leq 3^{-k^* p} + 3 \sqrt{\frac m n}\,.
\end{equation*}
On the other hand, $3^{-k^*p} = 9^p 3^{-(k^*+2)p} < 9^p \ep^p \leq 9^p \sqrt{\frac{m}{n}}$, so
\begin{equation*}
\E[W_p^p(\mu, \hat \mu_n)] \leq C_1 \sqrt{\frac m n}\,.
\end{equation*}

Taking the infimum over all possible choices of $\ep$ yields the bound.
\end{proof}

\subsection{Clusterable distributions}

We now return to the situation described in the introduction to this Section and analyze the case where $\mu$ is like a sum of Dirac masses.
This is the simplest example of where multiscale behavior can occur.
We validate the intuition presented above: when $\mu$ is approximately discrete, in the sense that it is supported on balls of small radius, then the convergence of $\hat \mu_n$ to $\mu$ enjoys the fast $n^{-1/2p}$ rate until $n$ is large even $\mu$ is absolutely continuous with respect to the Lebesgue measure.
We show that a similar phenomenon occurs when $\mu$ is the convolution of a discrete distribution with a small Gaussian, where we show that it is enough that most of the mass of $\mu$ is near a discrete distribution, even though the support is unbounded.

\begin{definition}
A distribution $\mu$ is $(m, \Delta)$-clusterable if $\supp(\mu)$ lies in the union of~$m$ balls of radius at most $\Delta$.
\end{definition}
Intuitively, the measure $\mu$ looks like a sum of $m$ Dirac measures at ``large scales,'' with high-dimensional information arriving only when we consider scales smaller than $\Delta$.

\begin{proposition}\label{prop:clusterable}
If $\mu$ is $(m, \Delta)$ clusterable, then for all $n \leq m (2 \Delta)^{-2p}$,
\begin{equation*}
\E[W^p_p(\mu, \hat \mu_n)] \leq (9^p + 3) \sqrt{\frac m n}\,.
\end{equation*}
\end{proposition}
\begin{proof}
Since $\supp(\mu)$ lies in the union of $m$ balls of radius at most $\Delta$, we have $\cN_{2 \Delta}(\mu) \leq m$.
Therefore if $n \leq m (2 \Delta)^{-2p}$, then
\begin{equation*}
m_n = \inf_{\ep > 0} \max\{\cN_\ep(\mu, \ep^p), n \ep^{2p}\} \leq \max\{\cN_{2 \Delta}(\mu), n (2 \Delta)^{2p}\} \leq m\,,
\end{equation*}
and the claim follows from Proposition~\ref{prop:sqrt_n_bound}.
\end{proof}

We can apply the above result to ``Diracs plus Gaussian'' case described in the introduction to this Section.
We first require a simple Lemma, which allows us bound the mass of a Gaussian outside of a small ball.

\begin{lemma}\label{lem:gaussian_tail}
If $Z \sim \cN(0, \Sigma)$, then for any $c \geq 5$,
\begin{equation*}
\p[\|Z\|^2_2 > c^2 \Tr(\Sigma)] \leq e^{-c^2/4}\,.
\end{equation*}
\end{lemma}
A proof appears in Appendix~\ref{appendix:omitted}.

Proposition~\ref{prop:clusterable} and Lemma~\ref{lem:gaussian_tail} yield the following claim.
\begin{proposition}
Let $\mu$ be a mixture of $m$ Gaussian distributions in $\RR^d$ equipped with the $\ell_2$ norm, and let $\sigma^2$ be an upper bound for the trace of the covariance matrix of each mixture component.
If $p \log \frac 1 \sigma \geq 25/4$, then for all $n \leq m (16 \sigma^2 p \log \frac 1 \sigma)^{-p}$,
\begin{equation*}
\E[W_p^p(\mu, \hat \mu_n)] \leq (9^p + 3) \sqrt{\frac m n}\,.
\end{equation*}
\end{proposition}
Since the measure $\mu$ is absolutely continuous with respect to the Lebesgue measure, if $d > 2p$, then asymptotically we have
\begin{equation*}
\E[W_p^p(\mu, \hat \mu_n)] \gtrsim n^{-p/d}\,,
\end{equation*}
which is slower than the $n^{-1/2}$ rate obtainable for small $n$.
\begin{proof}

Let $c = 2 \sqrt{p \log \frac 1 \sigma}$, which by assumption is at least $5$.
By Lemma~\ref{lem:gaussian_tail}, all but at most $e^{-c^2/4} = \sigma^{p}$ mass lies within balls of radius $c \sigma$ around the mixture centers.
Therefore $\cN_{2 c \sigma}(\mu, (2 c \sigma)^p) \leq \cN_{2 c \sigma}(\mu, \sigma^p) \leq m$.
If $n \leq m (16 \sigma^2 p \log \frac 1 \sigma)^{-p} = m (2 c \sigma)^{-2p}$, then we obtain
\begin{equation*}
m_n = \inf_{\ep > 0} \max\{\cN_\ep(\mu, \ep^p), n \ep^{2p}\} \leq \max\{\cN_{2 c \sigma}(\mu, (2 c\sigma)^p), n (2 c \sigma)^{2p}\} \leq m\,,
\end{equation*}
and the claim follows from Proposition~\ref{prop:sqrt_n_bound}.
\end{proof}
\subsection{Approximately low-dimensional sets}
We now broaden considerably to the general case where $\mu$ is supported on an approximately low-dimensional set.

\begin{definition}[{See~\cite{Tal95}}]
For any $S \subseteq X$, the \emph{$\ep$-fattening} of $S$ is
\begin{equation*}
S_\ep := \{y: D(y, S) \leq \ep\}\,.
\end{equation*}
\end{definition}
If $S' \subset S_\ep$ for some $S$, then $S'$ is close to $S$ in the sense that every point of $S'$ is within $\ep$ of some point in $S$.
In particular, $S' \subset S_\ep$ if the Hausdorff distance between $S'$ and $S$ is at most $\ep$.

Measures supported on $S_\ep$ are ``close'' to measures supported on $S$, and if $S$ is low-dimensional, then we obtain correspondingly better finite-sample rates.

\begin{proposition}\label{prop:approx_low_dim}
Suppose $\supp(\mu) \subseteq S_\ep$ for some $\ep >0$ and set $S$ satisfying
\begin{equation*}
\cN_{\ep'}(S) \leq (3 \ep')^{-d}
\end{equation*}
for all $\ep' \leq 1/27$ and for some $d > 2p$.
Then for all $n \leq (3\ep)^{-d}$,
\begin{equation*}
\E[W_p^p(\mu, \hat \mu_n)] \leq C_1 n^{-p/d}\,,
\end{equation*}
where
\begin{equation*}
C_1 = 27^p\left(2+\frac{1}{3^{\frac{d}{2} - p} -1}\right)\,.
\end{equation*}
\end{proposition}
In other words, $\hat \mu_n$ converges to $\mu$ at the $n^{-p/d}$ rate until $n$ is exponentially large in $d$.
In particular, if $\mu$ is absolutely continuous with respect to the Lebesgue measure on $\RR^{s}$ where $s \gg d$, then $\hat \mu_n$ converges to $\mu$ much faster initially (at the rate $n^{-p/d}$) than it does in the limit (at the rate $n^{-p/s}$).

\begin{proof}
Given any covering of $S$ by balls $B_1, \dots, B_m$ of diameter $\ep'$, the $\ep$-fattenings $(B_1)_\ep, \dots, (B_m)_\ep$ provide a covering of $S_\ep$ by balls of diameter $\ep' + 2 \ep$.
This implies for all $\ep' \geq \ep$ that
\begin{equation*}
\cN_{3 \ep'}(\mu) \leq \cN_{3 \ep'}(S_\ep) \leq \cN_{\ep' + 2 \ep}(S_\ep) \leq \cN_{\ep'}(S) \leq (3 \ep')^{-s}\,,
\end{equation*}
and hence that
\begin{equation*}
d_{\geq 3 \ep}(S_\ep) \leq s\,.
\end{equation*}
Therefore, if $n \leq (3 \ep)^{-d}$, then
\begin{equation*}
d_n \leq \max\left\{d_{\geq 3 \ep}(\mu), \frac{\log n}{- \log 3 \ep}\right\} \leq d\,.
\end{equation*}
The claim follows from Proposition~\ref{prop:finite_sample}.
\end{proof}

We can also relax the requirement that $\supp(\mu) \subseteq S_\ep$ to the statement that $\mu$ is concentrated near $S$.
\begin{proposition}\label{prop:concentrate_near_s}
Let $S$ be a set satisfying
\begin{equation*}
\cN_{\ep}(S) \leq (3 \ep)^{-d}
\end{equation*}
for all $\ep \leq 1/27$, for some $d > 2p$.
Suppose there exists a positive constant $\sigma$ such that $\mu$ satisfies
\begin{equation*}
\mu(S_\ep) \geq 1 - e^{-\ep^2/2 \sigma^2}
\end{equation*}
for all $\ep > 0$.
If $p \log \frac 1 \sigma \geq \frac{1}{18}$, then for all $n \leq \left(18p \sigma^2 \log \frac 1 \sigma\right)^{-d/2}$,
\begin{equation*}
\E[W_p^p(\mu, \hat \mu_n)] \leq C_1 n^{-p/d}\,,
\end{equation*}
where
\begin{equation*}
C_1 = 27^p\left(2+\frac{1}{3^{\frac{d}{2} - p} -1}\right)\,.
\end{equation*}
\end{proposition}
In other words, we again get the fast $n^{-p/d}$ rate until $n$ is of order approximately $\sigma^{-d}$.

\begin{proof}
Let $\ep = \sqrt{2 p \sigma^2 \log \frac 1 \sigma}$.
As in the proof of Proposition~\ref{prop:approx_low_dim}, the assumptions imply that
\begin{equation*}
d_{\geq 3 \ep}(S_\ep) \leq d\,.
\end{equation*}
Since
\begin{equation*}
\mu(S_\ep) \geq 1 - e^{-\ep^2/2 \sigma^2} = 1 - \sigma^p \geq 1 - (3 \ep)^p\,,
\end{equation*}
we conclude that as long as $n \leq \left(18p \sigma^2 \log \frac 1 \sigma\right)^{-d/2}$, then
\begin{equation*}
d_n \leq \max\left\{d_{\geq 3 \ep}(\mu, (3\ep)^p), \frac{\log n}{- \log 3 \ep}\right\} \leq d\,,
\end{equation*}
and the claim follows from Proposition~\ref{prop:finite_sample}.
\end{proof}

The condition appearing in Proposition~\ref{prop:concentrate_near_s} is notable because it resembles the guarantee of an isoperimetric inequality~\cite{ledoux2005concentration}.
Such inequalities are an important topic in modern geometric probability theory, and possess a close connection to the $W_1$ distance~\cite{BobGoe99}.

It is a striking fact about such inequalities that they are intimately connected to the concentration properties of $1$-Lipschitz functions.
\begin{proposition}[{See~\cite[Proposition~1.3]{ledoux2005concentration}}]\label{prop:lip_implies_iso}
Given a function $f:X \to \RR$, say that $m_f$ is a \emph{median} of $f$ if
\begin{equation*}
\p[f(X) \geq m_f] \geq 1/2 \quad \text{and} \quad \p[f(X) \leq m_f] \geq 1/2\,.
\end{equation*}
If for all $1$-Lipschitz functions $f: X \to \RR$ and medians $m_f$,
\begin{equation}\label{eqn:lipschitz_concentration}
\p[f(X) \geq m_f + t] \leq e^{-t^2/2\sigma^2}\,,
\end{equation}
then for any set $A$ with $\mu(A) \geq 1/2$,
\begin{equation}\label{eqn:isoperimetry}
\mu(A_\ep) \geq 1 - e^{-\ep^2/2 \sigma^2}\,.
\end{equation}
Conversely, if~\eqref{eqn:isoperimetry} holds for all sets $A$ with $\mu(A) \geq 1/2$, then~\eqref{eqn:lipschitz_concentration} holds for any $1$-Lipschitz function $f$ with median $m_f$.
\end{proposition}

The conditions of Proposition~\ref{prop:lip_implies_iso} have been used recently to show concentration bounds for the $W_1$ distance~\cite{BolGuiVil07}.
Here, we show that if $\mu$ possesses the property described above, then $\hat \mu_n$ enjoys a fast rate of convergence to $\mu$ for any $p$, as long as $\mu$ assigns a constant fraction of mass to a low-dimensional set.

\begin{proposition}
Suppose that $\mu$ satisfies either of the two equivalent conditions of Proposition~\ref{prop:lip_implies_iso} with some $\sigma$ satisfying $p \log \frac 1 \sigma \geq \frac{1}{18}$.
If $S$ is a set satisfying
\begin{equation*}
\cN_{\ep}(S) \leq (3 \ep)^{-d}
\end{equation*}
for all $\ep$, for some $d > 2p$ and $\mu(S) \geq 1/2$, then for all $n \leq \left(18p \sigma^2 \log \frac 1 \sigma\right)^{-d/2}$,
\begin{equation*}
\E[W_p^p(\mu, \hat \mu_n)] \leq C_1 n^{-p/d}\,,
\end{equation*}
where
\begin{equation*}
C_1 = 27^p\left(2+\frac{1}{3^{\frac{d}{2} - p} -1}\right)\,.
\end{equation*}
\end{proposition}
\begin{proof}
Combine Propositions~\ref{prop:concentrate_near_s} and~\ref{prop:lip_implies_iso}.
\end{proof}

\section{Concentration}\label{sec:concentration}
In addition to proving bounds on the expected value of the quantity $W_p^p(\mu, \hat \mu_n)$, we also show that it concentrates well around its expectation.
Previous work~\cite{BolGuiVil07,Boi11} has sought to obtain tail bounds of the form
\begin{equation*}
\p[W_p^p(\mu, \hat \mu_n) \geq t] \leq \psi_n(t)\,,
\end{equation*}
where $\psi_n(t)$ is some function exhibiting subgaussian decay.
The results of~\cite{BolGuiVil07} appear to obtain this rate, but the constants involved depend on $n$ and the ambient dimension of the space in a way that makes the results difficult to interpret.

We follow a different approach, which more clearly emphasizes the dependence of the tail on $n$ and the dimension.
The results of Sections~\ref{sec:asymptotic} and~\ref{sec:finite}, above, yield bounds on the expected value $\E[W_p^p(\mu, \hat \mu_n)]$.
As we have seen, the convergence of this quantity to~$0$ may be slow when the dimension is large.
On the other hand, we show below that, as long as $X$ is bounded, the quantity~$W_p^p(\mu, \hat \mu_n)$ concentrates well around its expectation independent of the dimension.
The argument is standard~\cite{Tal92} and is significantly easier to obtain than the above bounds on the expected value.

We require the following dual formulation~\cite{RusRac90,Rus91}.
\begin{definition}
Given a bounded continuous function $f: X \to \RR$, the \emph{$c$-transform} of $f$ (with respect to $D(\cdot, \cdot)^p$) is the function $f^c: X \to \RR$ defined by
\begin{equation*}
f^c(y) = \sup_{x \in X} (f(x) - D(x, y)^p)\,.
\end{equation*}
\end{definition}
The following claims are standard, and we provide a proof in Appendix~\ref{appendix:omitted} for completeness.
\begin{proposition}[Kantorovich duality]\label{prop:duality}
Given any pair of probability measures $\mu$ and $\nu$ on $X$ and any $p \in [1, \infty)$, the following duality holds:
\begin{equation}\label{eqn:dual_def}
W_p^p(\mu, \nu) = \sup_{f \in \cC_b(X)} \E_\mu f - \E_\nu f^c\,,
\end{equation}
where the supremum is taken over all bounded continuous functions on $X$ and $f^c$ is the $c$-transform of $f$ with respect to $D(\cdot, \cdot)^p$.
Moreover, if $\diam(X) \leq 1$, then we can take $0 \leq f(x) \leq 1$ for all $x \in X$.
\end{proposition}

We then obtain a concentration result via a standard bounded difference argument.
\begin{proposition}\label{prop:high_probability}
For all $n \geq 0$ and $0 \leq p < \infty$,
\begin{equation*}
\p[W_p^p(\mu, \hat \mu_n) \geq \E W_p^p(\mu, \hat \mu_n) + t] \leq \exp\left(-2nt^2\right)\,.
\end{equation*}
\end{proposition}
\begin{proof}
Let $\hat \mu_n$ be the empirical distribution corresponding to the i.i.d. samples~$X_1, \dots, X_n \sim \mu$.
We abbreviate $W_p^p(\mu, \hat \mu_n)$ by $W$.
By Proposition~\ref{prop:duality}, we can write
\begin{equation*}
W = \sup_{0 \leq f \leq 1} \E_{\hat \mu_n} f - \E_\mu f^c\,,
\end{equation*}
or, writing $W$ explicitly as a function of $X_1, \dots X_n$,
\begin{equation*}
W(X_1, \dots, X_n) = \frac 1 n \sup_{0 \leq f \leq 1}  \sum_{i=1}^n f(X_i) - \E_\mu f^c\,.
\end{equation*}

For any $x_1, \dots, x_n, x_n' \in X$, we have
\begin{alignat*}{2}
W(x_1, \dots, x_n) - W(x_1, \dots, x'_n) & = &\frac 1 n \bigg\{&\sup_{0 \leq f \leq 1}  \sum_{i=1}^{n} (f(x_i) - \E_\mu f^c) \\
&&-&\sup_{0 \leq f'\leq 1 } \sum_{i=1}^{n-1} (f'(x_i) -  \E_\mu {f'}^c) \\
&&+& f'(x_n')- \E_\mu {f'}^c\bigg\} \\
& \leq &\frac 1 n &\sup_{0 \leq f \leq 1} f(x_n) - f(x_n') \leq \frac 1 n\,.
\end{alignat*}
Applying McDiarmid's inequality~\cite{McD89} yields the bound.
\end{proof}

\section{Applications}
In this section, we sketch two applications of our work to machine learning and statistics.

\subsection{Quadrature}\label{sec:quadrature}
Numerical integration, or \emph{quadrature}, refers to the technique of approximating integrals by finite sums for the purpose of evaluating them at low computational cost.
Given a measure $\mu$, the goal is to choose points $x_1, \dots, x_n \in X$ and weights $\alpha_1, \dots, \alpha_n \in \RR_+$ such that the approximation
\begin{equation*}
\int f(x) \mathrm{d}\mu(x) \approx \sum_{i=1}^n \alpha_i f(x_i)
\end{equation*}
is as good as possible for a wide class of functions $f$.
This problem possesses close connections to optimal quantization, since the points $x_1, \dots, x_n$ naturally serve as a finite approximation to the underlying measure~\cite{GraLus07,Pag98}.

When only a single function $f$ is considered, one can show that a Monte Carlo method which chooses $x_1, \dots, x_n$ i.i.d. from $\mu$ with uniform weights $\alpha_i = n^{-1}$ for $1 \leq i \leq n$ is asymptotically suboptimal~\cite{Nov88}.
However, our results show that if we require that the approximation hold over the class of Lipschitz functions $\mathrm{Lip}(X)$, then this simple Monte Carlo scheme is asymptotically optimal for a wide class of measures.
\begin{proposition}
Denote by $\mathrm{Lip}(X)$ the class of $1$-Lipschitz on $X$
If $\mu$ is a measure supported on a regular set of dimension $d \geq 2$ and $\mu \ll \cH^d$, then for any $s > d$,
\begin{equation*}
\E\sup_{f \in \mathrm{Lip}(X)}\left| \int f(x) \mathrm{d}\mu(x) - \frac 1 n \sum_{i=1}^n f(X_i)\right| \lesssim n^{-1/s}\,,
\end{equation*}
where $X_1, \dots, X_n \sim \mu$ are independent.
On the other hand, for any $t < d$, $x_1, \dots, x_n \in X$, and $\alpha_i, \dots, \alpha_n \in \RR_+$,
\begin{equation*}
\sup_{f \in \mathrm{Lip}(X)}\left| \int f(x) \mathrm{d}\mu(x) -  \sum_{i=1}^n \alpha_i f(x_i)\right| \gtrsim n^{-1/t}\,.
\end{equation*}
\end{proposition}
\begin{proof}
Recalling~\eqref{eqn:w1_dual}, we immediately see that the first claim corresponds to the fact that $\E W_1(\mu, \hat \mu_n) \lesssim n^{-1/s}$, which follows from Corollary~\ref{cor:minkowski_upper_bound} and Proposition~\ref{prop:regular}.

For the second claim, we first note that by choosing $f(x) = c$ to be a constant function, we have
\begin{equation*}
\left| \int f(x) \mathrm{d}\mu(x) -  \sum_{i=1}^n \alpha_i f(x_i)\right| = c \left|1 - \sum_{i=1}^n \alpha_i\right|\,.
\end{equation*}
Since such an $f$ is Lipschitz, by choosing $c$ arbitrarily large we obtain that
\begin{equation*}
\sup_{f \in \mathrm{Lip}(X)}\left| \int f(x) \mathrm{d}\mu(x) -  \sum_{i=1}^n \alpha_i f(x_i)\right| = \infty
\end{equation*}
unless $\sum_{i=1}^n \alpha_i = 1$.
It therefore suffices to prove the claim when $\sum_{i=1}^n \alpha_i = 1$.
In this case, setting
\begin{equation*}
\nu = \sum_{i=1}^n \alpha_i \delta_{x_i}\,,
\end{equation*}
and again applying~\eqref{eqn:w1_dual} yields that this claim is equivalent to the fact that $W_1(\mu, \nu) \gtrsim n^{-1/t}$.
Since $\nu$ is supported on at most $n$ points, this follows from Corollary~\ref{cor:lower_bound} and Proposition~\ref{prop:regular}.
\end{proof}

\subsection{$k$-means clustering}
The authors of~\cite{GuiRos12} point out that many ``unsupervised learning'' techniques in machine learning involve constructing a simple approximation $\tilde \mu$ to a measure $\mu$ such that $W_2(\mu, \tilde \mu)$ is small.
One such example is the so-called $k$-means problem, where the goal is to find a set $S$ with $|S|\leq k$ minimizing the objective function
\begin{equation*}
\E D(X, S)^2\,,
\end{equation*}
where $X \sim \mu$.
It is not hard to see~\cite[Lemma~3.1]{GuiRos12} that this problem is equivalent to finding a measure $\tilde \mu$ supported on at most $k$ points such that $W_2(\mu, \tilde \mu)$ is as small as possible.
Given such a measure, we obtain a clustering of $\mu$ into at most $k$ pieces by constructing a Voronoi partition of $\supp(\mu)$ based on the $k$ points in $\supp(\tilde \mu)$.

The authors of~\cite{GuiRos12} show that for $k$ sufficiently large and for $X$ a compact, smooth $d$-dimensional manifold, it is possible to find a measure $\tilde \mu$ with $|\supp(\tilde \mu)| \leq k$ satisfying
\begin{equation*}
W_2(\mu, \tilde \mu) \leq C_1 \tau k^{-1/d} \text{ with probability $1 - e^{-\tau^2}$}
\end{equation*}
on the basis of $n = C_2 k^{2 + \frac 4 d}$ samples.
Corollary~\ref{cor:lower_bound} implies that this dependence on $k$ is asymptotically optimal.

Our results show that a much simpler procedure suffices in high dimensions.
As long as $d \geq 4$, the empirical measure $\hat \mu_k$ satisfies
\begin{equation*}
\E W_2(\mu, \hat \mu_k) \leq C_1' k^{-1/s}
\end{equation*}
for any $s > d$, and Proposition~\ref{prop:high_probability} then implies that
\begin{equation*}
W_2(\mu, \hat \mu_k) \leq C_1'' \tau k^{-1/s} \text{ with probability $1 - e^{-\tau^4}$}\,.
\end{equation*}

This shows that clustering a measure $\mu$ into $k$ pieces on the basis of $k$ i.i.d. samples from $\mu$ is asymptotically optimal, and enjoys concentration properties even better than the ones implied by~\cite{GuiRos12}.

\section{Conclusion and Future Work}
Our focus in this work has been to obtain sharper rates than previously available for the convergence of $\hat \mu_n$ to $\hat \mu$ in Wasserstein distance, both in asymptotic and finite-sample settings.
Our results give theoretical support to a phenomenon observed in practice: even though $W_p(\mu, \hat \mu_n)$ can converge very slowly for measures supported on a high-dimensional metric space, many measures arising in applications are intrinsically low dimensional, at least approximately, and therefore enjoy reasonably fast rates of convergence.

Our work leaves open whether slightly different versions of the Wasserstein distance can converge faster in general.
Recently, a version of the Wasserstein distance with an entropic penalty has been proposed and shown to have attractive theoretical properties and practical performance~\cite{SolGoeCut15,Cut13,CarDuvPeySch17,RolCutPey16}.
It is possible that these objects achieve better rates than the vanilla Wasserstein distance in the high-dimensional setting.

We also do not consider here empirical measures other than the simple $\hat \mu_n$.
In practice, a technique known as importance sampling~\cite{Buc13} is often used to reduce the variance of estimates produced on the basis of random samples from a distribution.
As noted in Section~\ref{sec:quadrature}, if $\mu$ is sufficiently regular, then no discrete measure on $n$ points can achieve better asymptotic performance than the empirical measure $\hat \mu_n$.
However, we conjecture that many reasonable sampling techniques should produce measures that are also asymptotically no worse than $\hat \mu_n$.
We leave this question for future work.

\section{Acknowledgments}

We thank Guillaume Carlier, Marco Cuturi, and Gabriel Peyr\'e for discussions related to this work.
JW would like to thank FB for his hospitality at INRIA, where this research was conducted.

\appendix
\input{WeeBac17a}
\newcommand{\etalchar}[1]{$^{#1}$}

\end{document}

%% file: WeeBac17a.tex
\section{Omitted proofs}\label{appendix:omitted}
\subsection{Proof of Proposition~\ref{prop:wp_dyadic_bound}}
We begin by giving an informal outline of the idea of the proof.

Consider a partition $\{Q_i\}_{i \in \cI}$ of $X$, for some index set $\cI$.
The measures $\mu$ and $\nu$ both induce measures on each set in the partition.
We will transport $\mu$ to $\nu$ by first moving mass \emph{between} sets in this partition, and then moving mass \emph{within} each set in the partition.
If $\mu(Q_i) \neq \nu(Q_i)$ for one of the sets $Q_i$, we we need to transport an amount of mass equal to $|\mu(Q_i) - \nu(Q_i)|$ into or out of $Q_i$.
In total, we can transport the mass that $\mu$ assigns to each set in the partition to its proper set under $\nu$ for a total cost of
\begin{equation*}
\sum_{i \in \cI} |\mu(Q_i) - \nu(Q_i)| \diam(S) \leq \sum_{i \in \cI} |\mu(Q_i) - \nu(Q_i)|\,,
\end{equation*}
where we use the fact that $\diam(S) \leq \diam(X) \leq 1$ by assumption.

After the first step of the transport plan, $\mu$ has been transported so that each set in the partition contains the correct total amount of mass.
It therefore suffices in the second step to properly arrange the mass \emph{within} each set.
Moving the mass within $Q_i$ cannot cost more than $\diam(Q_i)$, so the total cost of arranging the mass within each set is at most
\begin{equation*}
\sum_{i \in \cI} \nu(Q_i) \diam(Q_i) \leq \max_{i \in \cI} \diam(Q_i)\,.
\end{equation*}

We have obtained a transport of $\mu$ to $\nu$ for a total cost of approximately
\begin{equation*}
\max_{i \in \cI} \diam(Q_i) + \sum_{i \in \cI} |\mu(Q_i) - \nu(Q_i)|\,.
\end{equation*}

This ``single scale'' bound is generally not tight, but a more refined bound can be obtained by applying the above argument recursively: instead of naively bounding the cost of moving the mass within $Q_i$ by the quantity $\diam(Q_i)$, we can partition $Q_i$ into smaller sets  and estimate the cost of moving the mass within $Q_i$ by first moving it between the sets of the partition before moving it within each smaller set.
Iterating the argument $k^*$ times yields the bound.

We now show how to make the above argument precise.
Given two measures $\mu$ and $\nu$ on $X$, write $\cC(\mu, \nu)$ for the set of couplings between~$\mu$ and~$\nu$; that is, for the set of measures on $X \times X$ whose projection onto the first and second coordinate correspond to $\mu$ and $\nu$ respectively.

Fix a $k^* \geq 1$.
We will define two sequences of measure $\pi_k$ and $\rho_k$ on $X$ for $1 \leq k \leq k^*$ such that $\sum_{k=1}^{k^*} \pi_k \leq \mu$ and $\sum_{k=1}^{k^*} \rho_k \leq \nu$.
Given such a sequence, we set $\mu_1 = \mu$ and $\nu_1 = \nu$ and write
\begin{align*}
\mu_k  & = \mu - \sum_{\ell=1}^{k-1} \pi_\ell \\
\nu_k & = \nu - \sum_{\ell =1}^{k-1} \rho_\ell
\end{align*}
for $k \leq k^*+1$.

Note that if $\gamma_k \in \cC(\pi_k, \rho_k)$ for $1 \leq k \leq k^*$ and $\gamma_{k^*+1} \in \cC(\mu_{k^*+1}, \nu_{k^*+1})$, then
\begin{equation*}
\sum_{k=1}^{k^*+1} \gamma_k \in \cC\left(\sum_{k=1}^{k^*} \pi_k + \mu_{k^*+1}, \sum_{k=1}^{k^*} \rho_k + \nu_{k^*+1}\right) = \cC(\mu, \nu)\,,
\end{equation*}
therefore
\begin{equation*}
W_p^p (\mu, \nu) \leq \sum_{k=1}^{k^*} W_p^p(\pi_k, \rho_k) + W_p^p(\mu_{k^*+1}, \nu_{k^*+1})\,.
\end{equation*}

For $k \geq 1$, define
\begin{align*}
\pi_k &= \sum_{\substack{Q_i^k \in \cQ^k\\\mu_k(Q_i^k) > 0}} \left(1 - \frac{\nu_k(Q_i^k)}{\mu_k(Q_i^k)} \right)_+ \mu_k|_{Q_i^k}\,, \\
\rho_k &= \sum_{\substack{Q_i^k \in \cQ^k\\\nu_k(Q_i^k) > 0}} \left(1 - \frac{\mu_k(Q_i^k)}{\nu_k(Q_i^k)} \right)_+ \nu_k|_{Q_i^k}\,.
\end{align*}

Note that $0 \leq \pi_k \leq \mu_k$ and $0 \leq \rho_k \leq \nu_k$ for all $k$, hence $0 \leq \mu_k \leq \mu$ and $0 \leq \nu_k \leq \nu$ for all $k$ as well.

\begin{applemma}\label{lem:same_mass_pi_rho}
If $Q \in \cQ^{k-1}$, then
\begin{equation*}
\pi_{k}(Q) = \rho_{k}(Q)\,.
\end{equation*}
Moreover,
\begin{equation*}
\pi_k(S) = \rho_k(S) \leq \sum_{Q_i^k \in \cQ^k} |\mu(Q_i^k) - \nu(Q_i^k)|\,.
\end{equation*}
\end{applemma}


\begin{applemma}\label{lem:diameter_bound}
If $\alpha$ and $\beta$ are two measures on $X$ such that
\begin{equation*}
\alpha(Q) = \beta(Q)
\end{equation*}
for all $Q \in \cQ^{k}$, then
\begin{equation*}
W_p^p(\alpha, \beta) \leq \delta^{kp} \alpha(S)\,.
\end{equation*}
\end{applemma}
We can now obtain the final bound.
By Lemmas~\ref{lem:same_mass_pi_rho} and~\ref{lem:diameter_bound},
\begin{equation*}
W_p^p(\pi_k, \rho_k) \leq \delta^{(k-1)p} \sum_{Q_i^k \in \cQ^k} |\mu(Q_i^k) - \nu(Q_i^k)|
\end{equation*}
and
\begin{equation*}
W_p^p(\mu_{k^*+1}, \mu_{k^* + 1}) \leq \delta^{k^* p} \mu_{k^*+1}(S) \leq \delta^{k^* p} \mu(S) \leq \delta^{k^* p}\,.
\end{equation*}
The bound follows.
\qed

\subsection{Proof of Proposition~\ref{prop:ordering}}
We prove the inequalities in order.
If $d < d_H(\mu)$, then by~\cite[Proposition~10.3]{falconer1997techniques} there exists a compact set $K$ with positive mass and a $r_0 > 0$ such that
\begin{equation*}
\mu(B(x, r)) \leq r^d
\end{equation*}
for all $r \leq r_0$ and all $x \in K$.
(See also the proof of~\cite[Corollary~12.16]{GraLus07}.)
Let~$\tau < \mu(K)/2$.
If $S$ is any set with $\mu(S) \geq 1- \tau$, then $\mu(S \cap K) > \mu(K)/2$.
If $\cN_\ep(S) = N$, then in particular there exists a covering of $S \cap K$ by at most $N$ balls of radius $\ep$ whose centers all lie in $K$.
Indeed, any set of diameter at most $\ep$ which intersects $S \cap K$ is contained in a ball of radius $\ep$ whose center is in $K$.
If $\ep \leq r_0$, then each such ball satisfies $\mu(B(x, r)) \leq \ep^d$, so
\begin{equation*}
N \geq \ep^{-d} \mu(K)/2\,.
\end{equation*}
We therefore have for all $\tau$ sufficiently small,
\begin{equation*}
\liminf_{\ep \to 0} \frac{\log \cN_\ep'(\mu, \tau)}{- \log \ep} \geq d\,.
\end{equation*}
Thus $d_*(\mu) \geq d$.
Since $d < d_H(\mu)$ was arbitrary, we have $d_H(\mu) \leq d_*(\mu)$, as desired.

That $d_*(\mu) \leq d^*_p(\mu)$ follows from the simple observation that for all positive $\alpha$ and $\tau$,
\begin{equation*}
\liminf_{\ep \to 0} d_\ep(\mu, \tau) \leq \liminf_{\ep \to 0} d_\ep(\mu, \ep^{\alpha})\,.
\end{equation*}

Finally, if $d_M(\mu) \geq 2p$, then setting $s > d_M(\mu)$ yields
\begin{equation*}
\limsup_{\ep \to 0} d_\ep(\mu, \ep^{\frac{sp}{s - 2p}}) \leq \limsup_{\ep \to 0} d_\ep(\mu) = d_M(\mu) < s\,,
\end{equation*}
so $d_p^*(\mu) \leq s$.
Since $s > d_M(\mu)$ was arbitrary, we obtain $d_p^*(\mu) \leq d_M(\mu)$.
\qed
\subsection{Proof of Proposition~\ref{prop:covering_to_dyadic}}
Write $N_k = \cN_{3^{-(k+1)}}(S_k)$.
For $1 \leq k \leq k^*$, let $C^{k} = \{C^{k}_1, \dots\}$ be a finite covering of $S$ by balls of diameter $3^{-(k+1)}$ such that $C^{k}_1, \dots, C^{k}_{N_k}$ covers $S_k$.
Such a covering can always be found by choosing an optimal covering of $S_k$ and extending this covering to a covering of all of $S$.
Since $\cN_{3^{-(k^* + 1)}}(S) < \infty$, this requires only a finite number of additional balls.

We begin by constructing $\cQ^{k^*}$.
Let $\cQ^{k^*}_1 = C^{k^*}_1$, and for $1 < \ell \leq |C^{k^*}|$ let
\begin{equation*}
\cQ^{k^*}_\ell = C^{k^*}_\ell \setminus \left(\bigcup_{n = 1}^{\ell -1} \cQ^{k^*}_n \right)\,.
\end{equation*}
Let $\cQ^{k^*} = \{\cQ^{k^*}_1, \dots \}$.
Note that $\diam(\cQ^{k^*}_\ell) \leq \diam(C^{k^*}_\ell) = 3^{-(k^*+1)} < 3^{-k^*}$, that $\cQ^{k^*}$ forms a partition of $S$, and that at most $N_{k^*}$ elements of $\cQ^{k^*}$ intersect $S^{k^*}$.

We now show how to construct $\cQ^{k}$ from $\cQ^{k+1}$ and $C^k$.
Let
\begin{equation*}
\cQ^k_1 = \bigcup_{\substack{Q \in \cQ^{k+1} \\ Q \cap C^k_1 \neq \emptyset}} Q\,,
\end{equation*}
and for $1 < \ell \leq |C^{k^*}|$ let
\begin{equation*}
\cQ^k_\ell = \Big(\bigcup_{\substack{Q \in \cQ^{k+1} \\ Q \cap C^k_\ell \neq \emptyset}} Q\Big) \setminus \Big(\bigcup_{n = 1}^{\ell -1} \cQ^{k}_n \Big)\,.
\end{equation*}
Let $\cQ^{k} = \{\cQ^{k}_1, \dots \}$.

The sets in $\cQ^k$ clearly form a partition of $S$, and by construction at most $N_k$ elements of $\cQ^k$ intersect $S^k$
Moreover, since $\diam(C^k_\ell) \leq 3^{-(k+1)}$ for all $\ell$ and $\diam(Q) \leq 3^{-(k+1)}$ for all $Q \in \cQ^{k+1}$, the distance between any two points in $\cQ^k_\ell$ is at most $3 \cdot 3^{-(k+1)} = 3^{-k}$, so each element of $\cQ^k$ has diameter at most $3^{-k}$.
Finally, since each set in $\cQ^k$ is the union of sets in $\cQ^{k+1}$, the partition $\cQ^{k+1}$ refines $\cQ^k$, as desired.
\qed

\subsection{Proof of Proposition~\ref{prop:abs_cont}}
The only inequality that does not follow from Proposition~\ref{prop:ordering} is the first.
By absolute continuity, for all $\tau > 0$ there exists a $\sigma > 0$ such that any set $T$ for which $\mu(T) \geq 1- \tau$ satisfies $\cH^d(T) \geq \sigma$.
If $\cH^d(T) \geq \sigma$, then in particular for any covering $\{B(x_i, \ep)\}$ of $T$ by balls of radius $\ep$, we must have $\sum_{i} \ep^d \geq \sigma$.
Therefore such a covering contains at least $\sigma \ep^{-d}$ balls, so
\begin{equation*}
\frac{\log \cN_\ep(\mu, \tau)}{- \log \ep} \geq d + \frac{\log \sigma}{- \log \ep}\,,
\end{equation*}
and taking limits yields that $d_*(\mu) \geq d$, as desired.
\qed

\subsection{Proof of Proposition~\ref{prop:finite_sample_converse}}
For all integers $k \geq 0$, denote by $N_k$ the smallest positive integer $n$ such that $n$ is a power of two and $\delta_{n} \leq 2^{-k}$.
Such an integer always exists because the sequence $\delta_n$ decreases to $0$.
We require the following lemma, whose proof is deferred to Appendix~\ref{sec:proofs}.
\begin{applemma}\label{lem:converse_m_bound}
The sequence $N_{k+1}/N_k$ is bounded.
\end{applemma}
Let $m$ be an integer large enough that $N_{k+1}/N_k \leq 2^m$ for all $n$.
Let $\cQ$ be the standard dyadic partition of $[0, 1]$, with $\cQ^k$ being a partition of $[0, 1]^m$ consisting of~$2^{km}$ cubes of side length~$2^{-k}$.

Our measure $\mu$ will satisfy $\cN_{2^{-k}}(\mu) = N_{k-2}$ for all $k \geq 2$.
We will define a sequence of measures $\{\mu_k\}_{k=2}^{\infty}$ iteratively and construct $\mu$ as their limit in the weak topology.

Let $\mu_2$ be the uniform distribution on $[0, 1/4]^m$.
For each positive integer $k$, the measure $\mu_k$ will be supported on $N_{k-2}$ cubes in $\cQ$, and will be uniform on its support.
We will call a cube $Q_i \in \cQ^k$ \emph{live} if $\mu_k(Q_i) \neq 0$.

Fix an ordering $x_0, \dots, x_{2^m-1}$ of the $2^m$ elements of $\{0, 1\}^m$.
To produce $\mu_{k+1}$ from $\mu_k$, divide each live cube of $\mu_k$ into $2^m$ cubes of side length $2^{-(k+1)}$.
The ordering of $\{0, 1\}^m$ induces an order on these $2^m$ subcubes.

Given a live $Q \in \cQ^k$, define the restriction $\mu_{k+1}|_{Q}$ by requiring that $\mu_{k+1}(Q) = \mu_k(Q)$ and that $\mu_{k+1}|_Q$ be uniform on the union of the first $N_{k+1}/N_k$ subcubes of $Q$.
Note that $N_{k+1}/N_k$ is an integer because both $N_{k+1}$ and $N_k$ are powers of $2$, and by assumption $N_{k+1}/N_k \leq 2^m$, the total number of subcubes of $Q$.
Since $\cQ^k$ forms a partition of $[0, 1]^m$, combining the measures $\mu_{k+1}|_Q$ for $Q \in \cQ^k$ yields a probability measure $\mu_{k+1}$ on $[0, 1]^m$.
By Prokhorov's theorem, this sequence of measures $\mu_k$ possesses a subsequence converging in distribution to some measure $\mu$.

The following lemma collects necessary properties of $\mu$.
Its proof appears in Appendix~\ref{sec:proofs}.
\begin{applemma}\label{lem:mu_facts}
If $N_k \leq n < N_{k+1}$, then
\begin{equation*}
\cN_{2^{-k-4}}(\mu, 1/2) > n
\end{equation*}
Moreover,
\begin{equation*}
2^{-(k+1)} \leq \delta_n \leq 2^{-k}
\end{equation*}
and
\begin{equation*}
2^{-k -4} \leq n^{-1/d_n} \leq 2^{-k}\,.
\end{equation*}
\end{applemma}

We can now obtain the lower bound.
Let $\nu$ be any measure supported on at most $n$ points.
If $N_k \leq n < N_{k+1}$, then by Lemma~\ref{lem:mu_facts}, if $X \sim \mu$, then
\begin{equation*}
\p[\min_{y \in \supp(\nu)} \|X - y\|_\infty \leq 2^{-k-5}] < 1/2\,.
\end{equation*}
Markov's inequality therefore implies for any coupling $(X, Y)$ of $\mu$ and $\nu$ that
\begin{equation*}
\E[\|X - Y\|_\infty^p]^{1/p} \geq 2^{-k-4}\p[\min_{y \in \supp(\nu)} \|X - y\|_\infty > 2^{-k-5}]^{1/p} \geq 2^{-k-6} \geq 2^{-6}n^{-1/d_n}\,,
\end{equation*}
as claimed.
\qed

\subsection{Proof of Proposition~\ref{prop:duality}}
Both claims are standard, and details can be found in~\cite[Theorem~5.10]{Vil09}.
The first follows follows from the assumption that $X$ is a bounded Polish space.
For the second, we use the fact that the supremum is achieved by an $f$ satisfying
\begin{equation}\label{eqn:c_convexity}
f(x) = \inf_{y \in X} f^c(y) + D(x, y)^p \quad \forall x \in X\,.
\end{equation}
Let $f$ be a function achieving the supremum in~\eqref{eqn:dual_def} and satisfying~\eqref{eqn:c_convexity}.
By adding a constant to $f$ and $f^c$, we can assume that $\sup_{x \in X} f(x) = 1$.
Then for all $y \in X$,
\begin{equation*}
f^c(y) = \sup_{x \in X} f(x) - D(x, y)^p \geq 0\,,
\end{equation*}
and~\eqref{eqn:c_convexity} then implies
\begin{equation*}
f(x) \geq 0 \quad \forall x \in X\,,
\end{equation*}
as claimed.
\qed

\subsection{Proof of Lemma~\ref{lem:gaussian_tail}}
Let $\Sigma = \sum_{i=1}^d \lambda_i v_i v_i^\top$ be an eigendecomposition of $\Sigma$ with $\lambda_1 \geq \dots \geq \lambda_d \geq 0$.
Then $\|Z\|_2^2$ has the same distribution as $\sum_{i=1}^d \lambda_i \xi_i^2$, where $\xi_1, \dots, \xi_d$ are i.i.d. standard Gaussian random variables. 

By~\cite[Lemma~1]{LauMas00}, for any positive $t$, we have
\begin{equation*}
\p\left[\sum_{i=1}^d \lambda_i (\xi_i^2-1) \geq 2 \sqrt{ t \sum_{i=1}^d \lambda_i^2} + 2 \lambda_{1} t\right] \leq \exp(-t)\,.
\end{equation*}
Bounding both $\left(\sum_{i=1}^d \lambda_i^2\right)^{1/2}$ and $\lambda_1$ by $\Tr(\Sigma)$ yields
\begin{equation*}
\p\left[\|Z\|_2^2 \geq (1 + 2\sqrt t + 2 t)\Tr(\Sigma)\right] \leq \exp(-t)\,.
\end{equation*}
Finally, setting $c^2 = 2(\sqrt t + 1)^2$, we obtain
\begin{equation*}
\p\left[\|Z\|_2^2 \geq c^2\Tr(\Sigma)\right] \leq \exp(-(c - \sqrt 2)^2/2) \leq \exp(-c^2/4)
\end{equation*}
for all $c \geq 5$, as desired.
\qed

\section{Additional lemmas}\label{sec:proofs}
\subsection{Proof of Lemma~\ref{lem:same_mass_pi_rho}}
\begin{proof}
We first show that for any $\ell < k$, if $Q \in \cQ^{\ell}$, then
\begin{equation*}
\mu_k(Q) = \nu_k(Q)\,.
\end{equation*}

Suppose first that $Q \in \cQ^{k-1}$.
By definition, $\mu_k = \mu_{k-1} - \pi_{k-1}$.
We obtain
\begin{equation*}
\mu_k(Q) = (\mu_{k-1} - \pi_{k-1})(Q) = \min\{\mu_{k-1}(Q), \nu_{k-1}(Q)\}\,,
\end{equation*}
and likewise
\begin{equation*}
\nu_k(Q) = \min\{\mu_{k-1}(Q), \nu_{k-1}(Q)\}\,.
\end{equation*}

Since $\cQ$ is a dyadic partition, any $Q \in \cQ^\ell$ for $\ell < k$ can be written as a disjoint union of $Q_1, \dots, Q_m \in \cQ^{k-1}$.
Hence
\begin{equation*}
\mu_k(Q) = \sum_{i=1}^m \mu_k(Q_i) = \sum_{i=1}^m \nu_k(Q_i) = \nu_k(Q)\,,
\end{equation*}
as claimed.

Note that this also implies
\begin{equation*}
\pi_k(Q) = \mu_k(Q) - \mu_{k+1}(Q) = \nu_k(Q) - \nu_{k+1}(Q) = \rho_k(Q)\,.
\end{equation*}

We now prove the bound on $\pi_K(S)$.
By definition, 
\begin{equation*}
\rho_k(S) = \sum_{Q_i^k \in \cQ^k} (\nu_k(Q_i^k) - \mu_k(Q_i^k))_+ = \frac 1 2 \sum_{Q_i^k \in \cQ^k} |\nu_k(Q_i^k) - \mu_k(Q_i^k)|\,.
\end{equation*}

We now show that, for any $P \in \cQ^{k-1}$, there exist scalars $c_1, c_2 \in [0, 1]$ depending on $P$ such that
\begin{align*}
\mu_k|_P & = c_1 \mu|_P \\
\nu_k|_P & = c_2 \nu|_P\,.
\end{align*}

We proceed by induction on $k$.
By symmetry, it suffices to prove the claim for $\mu_k$ and $\mu$.
Since $\mu_1 = \mu$, it holds for $k = 1$.
Now assume $\mu_{k-1}|_P = c_1 \mu|_P$.
We have
\begin{equation*}
\mu_k|_P = \mu_{k-1}|_P - \pi_{k-1}|_P = \min\left\{\frac{\nu_{k-1}(P)}{\mu_{k-1}(P)}, 1\right\} \mu_{k-1}|_P = c_1' \mu|_P\,,
\end{equation*}
where $c_1' = \min\left\{\frac{\nu_{k-1}(P)}{\mu_{k-1}(P)}, 1\right\} c_1$.
This proves the claim.

Now, given such a $P \in \cQ^{k-1}$ and $c_1, c_2 \in [0, 1]$, we have $\mu_k(P) = \nu_k(P)$, so
\begin{equation*}
c_1 \mu(P) = c_2 \nu(P)\,.
\end{equation*}

Summing over the elements of $\cQ^k$ contained in $P$, we obtain
\begin{align*}
\sum_{Q_i^k \subset P} |\mu_k(Q_i^k) - \nu_k(Q_i^k)| & =\sum_{Q_i^k \subset P} |c_1 \mu(Q_i^k) - c_2\nu(Q_i^k)| \\
& \leq \sum_{Q_i^k \subset P} c_1 |\mu(Q_i^k) - \nu(Q_i^k)| + \sum_{Q_i^k \subset P} \nu(Q_i^k) |c_1 - c_2| \\
& = \sum_{Q_i^k \subset P} c_1 |\mu(Q_i^k) - \nu(Q_i^k)| + c_2 |\mu(P) - \nu(P)| \\
& \leq \sum_{Q_i^k \subset P} (c_1 + c_2) |\mu(Q_i^k) - \nu(Q_i^k)| \\
& \leq 2 \sum_{Q_i^k \subset P}|\mu(Q_i^k) - \nu(Q_i^k)|\,.
\end{align*}
Finally, summing over all $P \in \cQ^{k-1}$ yields
\begin{equation*}
\rho_k(S) = \frac 1 2 \sum_{Q_i^k \in \cQ^k} |\nu_k(Q_i^k) - \mu_k(Q_i^k)| \leq \sum_{Q_i^k \in \cQ^k}|\mu(Q_i^k) - \nu(Q_i^k)|\,,
\end{equation*}
as claimed.
\qed

\subsection{Proof of Lemma~\ref{lem:diameter_bound}}
Let
\begin{equation*}
\gamma = \sum_{\substack{Q_i^{k} \in \cQ^{k}\\\alpha(Q_i^{k})> 0}} \frac{\alpha \otimes \beta}{\alpha(Q_i^k)}\,.
\end{equation*}
Note that $\gamma \in C(\alpha, \beta)$.
Indeed, for any measurable $U \subset S$, since $\cQ^{k-1}$ is a partition of $S$, we have
\begin{equation*}
\gamma(S, U) = \sum_{Q_i^{k} \in \cQ^{k}} \frac{\alpha(Q_i^{k}) \beta(Q_i^{k} \cap U)}{\alpha(Q_i^{k-1})} = \beta(U)\,.
\end{equation*}
On the other hand, by assumption, $\alpha(Q_i^{k}) = \beta(Q_k^{k})$, so
\begin{equation*}
\gamma_k(U, S) = \sum_{Q_i^{k} \in \cQ^{k}} \frac{\alpha(Q_i^{k-1} \cap U) \beta(Q_i^{k-1})}{\beta(Q_k^{k-1})} = \alpha(U)\,.
\end{equation*}

We have
\begin{align*}
\int D(x, y)^p \mathrm{d}\gamma(x, y) & = \sum_{Q_i^{k} \in \cQ^{k}} \frac{1}{\alpha(Q_i^{k})}\int_{Q_i^{k}} D(x, y)^p \mathrm{d}\alpha(x) \mathrm{d}\beta(y) \\
&\leq \sum_{Q_i^{k} \in \cQ^{k}} \beta(Q_i^{k}) \diam(Q_i^{k})^p \\
&\leq  \alpha(S) \delta^{kp}\,.
\end{align*}
\end{proof}

\subsection{Proof of Lemma~\ref{lem:converse_m_bound}}
By assumption, there exist constants $c$ and $\alpha$ such that $\frac 1 c n^{\alpha} \leq \delta_n \leq c n^{\alpha}$ for all $n$ sufficiently large.
Let $M = (2 c^2)^{-1/\alpha}$.
Then for $n$ sufficiently large, 
\begin{equation*}
\delta_{M n} \leq c {M n}^{\alpha} = \frac{1}{2c} n^{\alpha} \leq \frac 1 2 \delta_n\,.
\end{equation*}
This implies that for $k$ sufficiently large, $\delta_{N_k} \leq 2^{-k}$ implies that $\delta_{M N_k} \leq 2^{-k-1}$, so that $N_{k+1} \leq M N_k$.
Hence $N_{k+1}/N_k \leq M$ for all $k$ sufficiently large, so $N_{k+1}/N_k$ is bounded for all $k$.
\qed

\subsection{Proof of Lemma~\ref{lem:mu_facts}}
We first show the key property of $\mu$.
For any $x \in [0, 1]^m$ and $r > 0$, denote by $B(x, r)$ the open $\ell_\infty$ ball of radius $r$ around $x$.
We claim that for any $x \in [0, 1]^m$ and $\ell \geq 2$,
\begin{equation*}
\mu(B(x, 2^{-\ell-1})) \leq \frac{1}{N_{\ell-2}}\,.
\end{equation*}

The claim certainly holds when $B(x, 2^{-\ell-1})$ exactly coincides with one of the cubes in $\cQ^\ell$, since each live cube in $\cQ^\ell$ has mass exactly $1/N_{\ell-2}$ by construction.

For all other $x$, note that the restriction of $\mu$ to each live cube in $\cQ^\ell$ is the same measure. In general, the cube $B(x, 2^{-\ell-1})$ intersects $2^m$ cubes cubes in $\cG_\ell$, so we can partition $B(x, 2^{-\ell-1})$ into $2^m$ pieces which, via translation, exactly cover a cube of~$\cQ^\ell$. Each piece has mass at most the mass of the corresponding piece in a live cube, hence the measure is at most the measure of a live cube.

This property immediately implies a bound on the number of balls needed to cover any set $S$ such that $\mu(S) \geq 1/2$.
Since each ball of diameter $2^{-\ell}$ has mass at most $1/N_{\ell-2}$, to cover a set of mass $1/2$ requires at least $N_{\ell-2}/2$ balls.
Therefore for all  $\ell \geq 2$,
\begin{equation}\label{eqn:ball_bound}
\cN_{2^{-\ell}}(\mu, 1/2) \geq N_{\ell-2}/2\,.
\end{equation}

Since $n \leq N_{k+1}$, we have by definition $\delta_n > 2^{-(k+1)}$.
Because $\frac{\log n}{- \log \delta_n}$ is nondecreasing and at least $1$ for all $n \geq 2$, we have as long as $n \geq 2$ that
\begin{equation*}
\frac{\log 2n}{- \log \delta_{2n}} \geq \frac{\log n}{- \log \delta_n} \geq \frac{\log 2n}{- \log \delta_n/2}\,,
\end{equation*}
and therefore $\delta_{2n} \geq \frac 1 2 \delta_n > 2^{-(k+2)}$.
This implies $N_{k+2}/2 > n$.

Choosing $\ell = k+4$ in~\eqref{eqn:ball_bound} yields
\begin{equation*}
\cN_{2^{-k - 4}}(\mu, 1/2) > n\,.
\end{equation*}
This proves the first claim.

If $N_k \leq n < N_{k+1}$, then by definition of $N_k$ and $N_{k+1}$ and the fact that $\delta_n$ is nonincreasing in $n$,
\begin{equation*}
\delta_n \leq \delta_{N_k} \leq 2^{-k}
\end{equation*}
and
\begin{equation*}
\delta_n > 2^{-k-1}\,.
\end{equation*}
This proves the second claim.

To prove the third claim, we first note that the definition of $d_n$ implies that
\begin{equation*}
n^{-1/d_n}
\end{equation*}
is nonincreasing as $n$ increases.
We can therefore prove an upper bound on $n^{-1/d_n}$ by proving an upper bound on $N_k^{-1/d_{N_k}}$.

Recall that
\begin{align*}
d_{N_k} = \inf_{\ep > 0} \max \left\{d_{\geq \ep}(\mu, \ep^p), \frac{\log N_k}{- \log \ep}\right\}\,.
\end{align*}
Choosing $\ep = 2^{-(k+2)}$ yields
\begin{equation*}
d_{N_k} \leq \max \{d_{\geq 2^{-(k+2)}}(\mu), \frac{\log_2 N_k}{k+2}\}\,.
\end{equation*}
To bound the first term, note that if $\ep' \in [2^{-\ell}, 2^{-\ell+1})$, then $\cN_{\ep'}(\mu) \leq \cN_{2^{-\ell}}(\mu) = \cN_{\ell-2}$.
Therefore $d_\ep' = \frac{\log \cN_{\ep'}(\mu)}{- \log \ep'} \leq \frac{\log N_{\ell -2}}{\ell - 1}$.

As above, since $\frac{\log n}{- \log \delta_n}$ is non decreasing and at least $1$ for all $n \geq 2$, we have that $\delta_{N_{\ell-2}} \geq \frac 1 2 \delta_{N_{\ell-2}/2} > 2^{-\ell + 1}$.
Combining this with the above bound implies that if $\ep' \in [2^{-\ell}, 2^{-\ell +1})$, then
\begin{equation*}
d_\ep' \leq \frac{\log N_{\ell - 2}}{- \log \delta_{N_{\ell -2}}}\,.
\end{equation*}
The assumption that $\frac{\log n}{- \log \delta_n}$ is nonincreasing therefore implies
\begin{equation*}
d_{\geq 2^{-{k+2}}}(\mu) \leq \max_{2 \leq \ell \leq k+2} \frac{\log N_{\ell - 2}}{- \log \delta_{N_{\ell -2}}} \leq \frac{\log N_k}{- \log \delta_{N_k}} \leq \frac{\log_2 N_k}{k}\,.
\end{equation*}
We obtain
\begin{equation*}
d_{N_k} \leq \frac{\log_2 N_k}{k}\,,
\end{equation*}
so $n^{-1/d_n} \leq N_k^{-1/d_{N_k}} \leq 2^{-k}$.

To obtain the lower bound, note that if $\ep \leq 2^{-(k+4)}$, then
\begin{equation*}
d_{\geq \ep}(\mu, \ep^p) \geq d_{2^{-(k+4)}}(\mu, 1/2) > \frac{\log_2 n}{k+4}\,,
\end{equation*}
where we have used the fact proved above that $\cN_{2^{-(k+4)}}(\mu, 1/2) > n$.
If $\ep > 2^{-(k+4)}$, then
\begin{equation*}
\frac{\log n}{- \log \ep} > \frac{\log_2 n}{k+4}\,.
\end{equation*}
Combining these bounds yields
\begin{equation*}
d_n = \inf_{\ep > 0} \max \left\{d_{\geq \ep}(\mu, \ep^p), \frac{\log n}{- \log \ep}\right\} > \frac{\log_2 n}{k+4}\,,
\end{equation*}
so
\begin{equation*}
n^{-1/d_n} > 2^{-(k+4)}\,,
\end{equation*}
as claimed.
\qed